\documentclass[11pt, letterpaper]{article}
\usepackage[utf8]{inputenc}
\usepackage[letterpaper, margin=1in]{geometry}
\usepackage[T1]{fontenc}
\usepackage{hyperref}

\usepackage{amsmath,amsthm,amssymb}
\usepackage[capitalise, noabbrev]{cleveref}
\usepackage[mathscr]{eucal}
\usepackage[left,mathlines]{lineno}
\usepackage{lmodern}
\usepackage{fontawesome5}
\usepackage{microtype}
\usepackage{textcomp}
\usepackage{thm-restate}
\theoremstyle{plain}
\newtheorem{theorem}{Theorem}
\newtheorem{lemma}[theorem]{Lemma}
\newtheorem{corollary}[theorem]{Corollary}

\newtheorem{example}[theorem]{Example}

\newtheorem{observation}[theorem]{Observation}
\newtheorem{claim}[theorem]{Claim}
\newenvironment{claimproof}[1][\proofname]{\begin{proof}[Claim Proof]}{\end{proof}}

\newcommand{\aemail}[1]{\href{mailto:#1}{\faIcon[regular]{envelope}~\nolinkurl{#1}}}
\newcommand{\aorcid}[1]{\href{https://orcid.org/#1}{\faOrcid}}
\newcommand{\ahome}[1]{\href{#1}{\faHome}}
	
\title{Arboricity Nearly Bounds Degeneracy}

\author{
Michał Lasoń \thanks{Institute of Mathematics of the Polish Academy of Sciences, ul.\ Śniadeckich 8, 00-656 Warszawa, Poland, 
\aemail{michalason@gmail.com}~\aorcid{0000-0003-4830-2270},
Supported by grant no.~2019/34/E/ST1/00087 from National Science Centre, Poland.}
\and
Bartłomiej Bosek \thanks{Institute of Theoretical Computer Science, Faculty of Mathematics and Computer Science, Jagiellonian University, Kraków, Poland,
\aemail{bartlomiej.bosek@uj.edu.pl}~\aorcid{0000-0001-8756-3663},
Partially supported by grant no.~2019/35/B/ST6/02472 from National Science Centre, Poland.}
\and
Grzegorz Gutowski \thanks{Institute of Theoretical Computer Science, Faculty of Mathematics and Computer Science, Jagiellonian University, Kraków, Poland,
\aemail{grzegorz.gutowski@uj.edu.pl}~\ahome{https://grzegorz.gutowscy.pl}~\aorcid{0000-0003-3313-1237},
Partially supported by grant no.~2023/49/B/ST6/01738 from National Science Centre, Poland.}
\and
Jakub Przybyło \thanks{AGH University of Krakow, Faculty of Applied Mathematics, al.\ A.\ Mickiewicza 30, 30-059 Kraków, Poland, 
\aemail{jakub.przybylo@agh.edu.pl}~\aorcid{0000-0002-1262-7017},
Partially supported by the AGH University of Krakow under grant no. 16.16.420.054 and Excellence Initiative – Research University AGH University of Krakow, funded by the Polish
Ministry of Science and Higher Education.}
}
\date{}

\usepackage[disableredefinitions]{complexity}
\usepackage{nicefrac}
\usepackage{tcolorbox}
\usepackage{subcaption}
\usepackage{tikz}
\usetikzlibrary{arrows,calc,math,patterns,shapes,fit,decorations.pathreplacing,shapes.geometric}
\usepackage{todonotes}
\usepackage{xcolor}
\usepackage{xspace}

\definecolor{defblue}{rgb}{0.121,0.47,0.705}
\definecolor{linkblue}{rgb}{0.098,0.098,0.4392}
\definecolor{newred}{rgb}{0.9,0.3,0.3}
\let\emph\relax
\DeclareTextFontCommand{\emph}{\color{defblue}\em}

\newcommand{\defproblem}[4]{
  \begin{tcolorbox}%
    \hspace{0ex}\hspace*{-2.8ex}
    \begin{minipage}{0.99\textwidth}
      \vspace{0ex}\vspace*{-1ex}
      \begin{tabular}{@{}l@{~~}p{0.9\textwidth}@{}}
        {\sf\bfseries\color{gray} Problem:} & #1\\[.1ex]
        {\sf\bfseries\color{gray} Input:} & #2\\[.1ex]
        {\sf\bfseries\color{gray} #4:} & #3\\[-1ex]
      \end{tabular}
    \end{minipage}
  \end{tcolorbox}
}

\newcommand{\defdecproblem}[3]{\defproblem{#1}{#2}{#3}{Question}}
\newcommand{\defoptproblem}[3]{\defproblem{#1}{#2}{#3}{Output}}
\newcommand{\OPdecomposition}{\textsc{Decomposition}\xspace}

\newcommand{\DPpositive}{\textsc{PositiveObject}\xspace}
\newcommand{\DPdec}[2]{\textsc{$(#1,#2)$\nobreakdash-Decomposition}\xspace}
\newcommand{\DPham}{\textsc{HamiltonianCycle}\xspace}
\newcommand{\DPfixham}[1]{\textsc{HamiltonianCycleFixedEdge[$#1$-Regular]}\xspace}

\newcommand{\yes}{\textsc{Yes}}

\let\geq\geqslant
\let\le\leqslant
\let\ge\geqslant

\let\rho\varrho

\newcommand{\brac}[1]{{\left(#1\right)}}
\newcommand{\bbrac}[1]{{\Big(#1\Big)}}

\newcommand{\set}[1]{\left\{#1\right\}}

\newcommand{\norm}[1]{{\left|#1\right|}}

\newcommand{\Oh}[1]{O\brac{#1}}
\newcommand{\oh}[1]{o\brac{#1}}

\DeclareMathOperator{\score}{score}
\DeclareMathOperator{\rank}{rank}

\DeclareMathOperator{\redc}{r}

\DeclareMathOperator{\pot}{f}
\DeclareMathOperator{\Pot}{f}

\begin{document}

\maketitle

\begin{abstract}

Arboricity and degeneracy are two fundamental and closely related graph parameters that measure the sparsity of a graph. Every $k$\nobreakdash-degenerate graph is $k$\nobreakdash-arboric, but some $k$\nobreakdash-arboric graphs are only $(2k-1)$\nobreakdash-degenerate.
However, every maximal $k$\nobreakdash-arboric multigraph with $n$ vertices and every maximal $k$\nobreakdash-degenerate multigraph with $n$ vertices has exactly $k(n-1)$ edges.
These basic observations lead to a natural structural question: 
\begin{center}
\textit{How far are $k$\nobreakdash-arboric graphs from being $k$\nobreakdash-degenerate?}
\end{center}
We answer this question by showing that:
\begin{center}
\textit{By at most a $(k-1)$\nobreakdash-bounded-degree graph apart.}
\end{center}

More specifically, we prove that a $k$\nobreakdash-arboric multigraph admits a $(k,k-1)$\nobreakdash-decomposition, that is, its edges can be partitioned into two multisets such that one spans a $k$\nobreakdash-degenerate multigraph and the other spans a multigraph with every vertex having degree at most $k-1$.

Moreover, we provide a complete characterisation of all possible such decomposition types. Namely, for any integers $k \ge 1$ and $d,h \ge 0$ we show that every $k$\nobreakdash-arboric multigraph admits a $(d,h)$\nobreakdash-decomposition if and only if 
	$d\geq k$ and $d+h\geq 2k-1$. 

Our proofs are constructive and we present a polynomial time algorithm that produces such decompositions.
By contrast, we show that related decision problems for general graphs (without constraints on the arboricity) are \NP-complete.
\end{abstract}

\section{Introduction}

Throughout this paper, by a \emph{graph} we mean a finite loopless multigraph. For simplicity, 
we use the term set of edges instead of multiset of edges, although edge multiplicities are always taken into account.
A graph is \emph{$k$\nobreakdash-degenerate} if every induced subgraph contains a vertex of degree at most $k$.
Equivalently, a graph is $k$\nobreakdash-degenerate if the set of vertices admits a linear order in which every vertex has at most $k$ edges to neighbors that are earlier in the order.
A graph is \emph{$k$\nobreakdash-arboric} if the set of edges admits a partition into $k$ forests.
Equivalently, the celebrated Nash-Williams Theorem~\cite{NashW1961,NashW1964} (see also \cite{Seymour1998,Lason2015} for interesting matroid generalisations) states that a graph is $k$\nobreakdash-arboric if every subset of $\ell$ vertices spans at most $k(\ell-1)$ edges.
Notice that a graph is $1$\nobreakdash-arboric if and only if it is $1$\nobreakdash-degenerate if and only if it is a forest.
Both $k$-degenerate and $k$-arboric graphs can have vertices with arbitrarily high degree.
A graph is \emph{$k$\nobreakdash-bounded-degree} if every vertex has degree at most $k$.

It is a basic observation that every $k$\nobreakdash-degenerate graph is $k$\nobreakdash-arboric.
To see that, one can process the vertices in the order guaranteed by the degeneracy condition -- for each vertex, color the edges to neighbors that are earlier in the order (there are at most $k$ such edges) using different colors from the set $\{1,\dots,k\}$.
After processing all the vertices, edges in every single color span a $1$\nobreakdash-degenerate graph (a forest), and as a result we obtain a partition of the edge set into $k$ forests.

However, a $k$-arboric graph is not necessarily $k$-degenerate.
It is $k$-degenerate if and only if its edge set admits a partition into $k$ graphs that are $1$\nobreakdash-degenerate (forests) with the same linear order. 
But, as there are at most $k(n-1)$ edges in a $k$-arboric graph on $n$ vertices, there is always a vertex with degree less than $2k$.
Thus, we obtain that every $k$\nobreakdash-arboric graph is $(2k-1)$\nobreakdash-degenerate.
Moreover, the complete graph $K_{2k}$ shows that this bound is tight.
For another illustration of this difference, notice that planar simple graphs are $3$-arboric and $5$-degenerate, and these bounds are also tight.
On the other hand, every maximal $k$\nobreakdash-arboric graph as well as every maximal $k$\nobreakdash-degenerate graph on $n$ vertices has exactly $k(n-1)$ edges.

In general, degeneracy is more often used as a basis for various structural arguments or effective algorithms, than arboricity is.
In particular, it is often the case that arguments for $k$-arboric graphs simply trade a multiplicative factor of $2$ and reduce to the case of $(2k-1)$\nobreakdash-degenerate graphs.
These theoretical disparities as well as application aspects lead to a natural open-ended question: 

\begin{center}
\textit{How far are $k$\nobreakdash-arboric graphs from being $k$\nobreakdash-degenerate?}
\end{center}
Specifically, what are the strongest constraints that one can impose on the edges removed from a $k$\nobreakdash-arboric graph in order to get a $k$\nobreakdash-degenerate graph?
We answer this question by showing that:
\begin{center}
\textit{Every $k$-arboric graph is $k$-degenerate up to removing edges of a $(k-1)$\nobreakdash-bounded-degree subgraph.}
\end{center}
As every $k$-arboric graph is $(2k-1)$-degenerate up to removing edges of a $0$\nobreakdash-bounded-degree subgraph, we show also that there is a gradual change between these two extremal cases:
\begin{center}
\textit{Every $k$-arboric graph is $(k+\ell)$-degenerate up to removing a $(k-\ell-1)$\nobreakdash-bounded-degree subgraph.}
\end{center}

These results are best possible, as shown in the next subsection.
To the best of our knowledge, these are first decomposition results of this kind.
Analogous decompositions were previously known only in restricted planar settings -- for triangle-free planar simple graphs (i.e., when $k=2$)~\cite{XuZ2025} and for planar simple graphs (i.e., when $k=3$)~\cite{ChoCKPSZ2022}.

Our proof is a complex inductive reasoning for which we extend the family of considered combinatorial objects and prove a more general statement.
The arguments are constructive and we present a polynomial time algorithm that produces desired decompositions.
We also show that related decision problems for general graphs (without constraints on the arboricity) are \NP-complete. We conjecture that there are sharp complexity thresholds for these problems as soon as we drop the constraints on the arboricity.

\subsection{Results}

A \emph{$(d,h)$\nobreakdash-decomposition} of a graph $G=(V,E)$ is a partition of the edge set $E$ of~$G$ into two subsets $E=E_1 \sqcup  E_2$ such that the graph $G_1=(V,E_1)$ is $d$\nobreakdash-degenerate and the graph $G_2=(V,E_2)$ is $h$\nobreakdash-bounded-degree. We call this a decomposition of \emph{type $(d,h)$}. The main result of the paper is a complete  characterisation of all possible decomposition types that exist for all $k$\nobreakdash-arboric multigraphs.

\begin{restatable}{theorem}{thmsuper}\label{thm:super}
For any integers $k \ge 1$, $d,h \ge 0$ inequalities $d\geq k$ and $d+h\geq 2k-1$ are necessary and sufficient conditions for existence of a $(d,h)$\nobreakdash-decomposition for every $k$\nobreakdash-arboric multigraph. \newline
In particular, every $k$\nobreakdash-arboric multigraph admits a $(k+\ell,k-\ell-1)$\nobreakdash-decomposition for $k-1\ge \ell \ge 0$.
\end{restatable}


\cref{thm:super} consists of two implications -- easy (inequalities are necessary for a $(d,h)$\nobreakdash-decomposition) and hard (inequalities are sufficient). At the end of this subsection we give two examples that justify the easy implication. The fundamental case for the hard implication is the one for $(d,h)=(k,k-1)$. Thus, we state this special case as a separate theorem and most of the paper is devoted to its proof.

\begin{restatable}{theorem}{thmmain}\label{thm:main}
For every integer $k\ge1$, a $k$\nobreakdash-arboric multigraph admits a $(k,k-1)$\nobreakdash-decomposition.
\end{restatable}

A proof of \cref{thm:main} is given in \cref{sec:main}. Next, in \cref{sec:super}, we show the following lemma which allows to gradually change a $(k,k-1)$\nobreakdash-decomposition and obtain 
all $(k+\ell,k-\ell-1)$\nobreakdash-decompositions.

\begin{restatable}{lemma}{lemsuper}\label{lem:super}
For any integers $d \geq 0$, $h \geq 1$, and a $(d,h)$\nobreakdash-decomposition of a multigraph, there is a $(d+1,h-1)$\nobreakdash-decomposition of the same multigraph.
\end{restatable}

Our proofs are constructive and produce aforementioned decompositions in polynomial time. In \cref{sec:algorithm} we give an effective implementation and prove the following lemma.
\smallskip

\begin{restatable}{lemma}{lemalgorithm}\label{lem:algorithm}
For integers $k \ge 1$, $d,h \ge 0$ satisfying $d\geq k$ and $d+h\geq 2k-1$ there is an algorithm that constructs a $(d,h)$\nobreakdash-decomposition of a $k$\nobreakdash-arboric multigraph $G$ with $m$ edges in time $\Oh{m^3}$.
\end{restatable}

By contrast, in \cref{sec:hardness}, we show that finding $(d,h)$\nobreakdash-decompositions for general graphs without a constraint on the arboricity is \NP-hard. We formalize this result in the following theorem.
\smallskip

\defdecproblem
{\DPdec{d}{h}}
{A graph $G$}
{Does $G$ admit a $(d,h)$\nobreakdash-decomposition?}

\begin{restatable}{theorem}{thmhard}\label{thm:hard}
For any integers $d \geq 1$, $h \geq 1$, the problem \DPdec{d}{h} is \NP-complete.
\end{restatable}

Below we present two examples that show the easy implication in \cref{thm:super}. Notice that these examples are simple graphs, therefore \cref{thm:super} provides also a characterisation of all possible decomposition types for all k-arboric simple graphs.


\begin{example}\label{exm:d}
For any integers $k \ge 1$, $h \ge 0$ the complete bipartite graph $K_{k,kh+h+k}$ is $k$\nobreakdash-arboric and does not admit a $(k-1,h)$\nobreakdash-decomposition.
\end{example}

Indeed, call $K_{k,kh+h+k}$ shortly $G$ and denote its bipartition classes by $X,Y$ where $\norm{X}=k$ and $\norm{Y}=kh+h+k$. It is easy to verify that $G$ is $k$\nobreakdash-arboric.
Suppose to the contrary, that there is a $(k-1,h)$\nobreakdash-decomposition of $G$ into a $h$\nobreakdash-bounded-degree graph $G_2$, and a graph $G_1$ with $(k-1)$\nobreakdash-degenerate order $\ll$.
Let $x$ be the vertex in $X$ that appears last in $\ll$.
As $x$ is of degree $kh+h+k$ in~$G$, $x$ is of degree at least $kh+k$ in~$G_1$, and since $\ll$ is $(k-1)$\nobreakdash-degenerate, there are at least $kh+k-(k-1)=kh+1$ vertices in $Y$ that appear later than $x$ in $\ll$.
For any vertex $y \in Y$ with $x \ll y$, we have that $y$ is of degree $k=(k-1)+1$ in~$G$ and has all the neighbors appearing earlier in the $(k-1)$\nobreakdash-degenerate order $\ll$.
Thus, at least one edge incident to $y$ is in $G_2$.
We get that there are at least $kh+1$ edges in $G_2$.
But, as $\norm{X}=k$ and $G_2$ is $h$\nobreakdash-bounded-degree, there are at most $kh$ edges in $G_2$ -- a contradiction.


\begin{example}\label{exm:dh}
For any integers $k \ge 1$, $d,h \ge 0$ the complete graph $K_{2k}$ is $k$\nobreakdash-arboric and does not admit a $(d,h)$\nobreakdash-decomposition with $d+h<2k-1$.
\end{example}

The $K_{2k}$ is well-known to be $k$\nobreakdash-arboric. If it has a $(d,h)$\nobreakdash-decomposition, then the last vertex in the  $d$\nobreakdash-degenerate order has degree at most $d+h$. As the degree of every vertex is $2k-1$, we get that $d+h\ge 2k-1$.

\subsection{Related Work}

The first nontrivial case of \cref{thm:main} is the case for $k=2$ 
and it already turns out to be hard.
In this case, \cref{thm:main} states that every $2$-arboric graph decomposes into a $2$-degenerate graph and a matching.
In a private communication, Grytczuk informed us that the question of whether \cref{thm:main} holds for $k=2$ was asked already in 1998 by Hałuszczak.
Some unpublished progress towards it was reported to him by Borowiecki and Sidorowicz, who worked on this problem with Enomoto and Hagita.
In full generality the problem for $k=2$ remained open, but an analogous decomposition was shown to exist for triangle-free planar simple graphs by Xu and Zhu~\cite{XuZ2025}.
Every triangle-free planar simple graph is $2$\nobreakdash-arboric, thus it follows from \cref{thm:main}.
The same follows for outerplanar simple graphs as it is another important family of planar graphs that are $2$\nobreakdash-arboric.

A series of papers~\cite{LihSWZ2001,DongX2009,LiLWZ2023} show that for every $m\in\set{5,6,7,8,9}$, a planar simple graph with no subgraph isomorphic to $C_4$ or $C_m$ admits a $(2,1)$\nobreakdash-decomposition.
It is an easy calculation using Euler's formula and the fact that in such a graph there is no $4$\nobreakdash-face, no two $3$\nobreakdash-faces are adjacent, and every $5$\nobreakdash-face is adjacent to at most three $3$\nobreakdash-faces, which shows that any such graph is $2$\nobreakdash-arboric.
Therefore, \cref{thm:main} gives a $(2,1)$\nobreakdash-decomposition for these graphs as well.

For $k=3$, a $(3,2)$\nobreakdash-decomposition and a $(4,1)$\nobreakdash-decomposition were shown to exist for every planar simple graph by Cho, Choi, Kim, Park, Shan, and Zhu~\cite{ChoCKPSZ2022}.
As planar simple graphs are $3$\nobreakdash-arboric, these results follow from \cref{thm:super}.
Other families of graphs that are known to be $3$\nobreakdash-arboric are $K_5$\nobreakdash-minor-free simple graphs and projective planar simple graphs, see~\cite{MukaeOST2022}.

For the particular case $(d,h)=(1,1)$, a graph that admits a $(1,1)$\nobreakdash-decomposition is said to be \emph{matching decyclable}.
Lima, Rautenbach, Souza, and Szwarcfiter~\cite{LimaRSS2017}, and then Protti and Souza~\cite{ProttiS2018} considered this concept and have shown that deciding if a graph is matching decyclable is \NP-complete.
Their reduction is the basis for the proof of \cref{thm:hard}.

\subsection{Applications}

One particular application of $(d,h)$\nobreakdash-decompositions is in the area of generalized colorings.
We say that a vertex coloring is \emph{$h$\nobreakdash-defective} if each color class spans an $h$\nobreakdash-bounded-degree subgraph.
The concept of $h$\nobreakdash-defective colorings was introduced by \v{S}krekovski~\cite{Skrekovski1999} and independently by Eaton and Hull~\cite{EatonH99}; see a survey by Wood~\cite{Wood2018}.
For any given graph~$G$, if $G$ admits a decomposition into a $d$\nobreakdash-degenerate graph $G_1$ and a $h$\nobreakdash-bounded-degree graph $G_2$, then the greedy proper coloring of $G_1$ produces an $h$\nobreakdash-defective coloring of~$G$ using at most $d+1$ colors.
Thus, \cref{thm:super} gives the following corollary.

\begin{corollary}\label{cor:defective}
For any integers $k \ge 1$ and $k\ge \ell \ge 1$ a $k$\nobreakdash-arboric multigraph admits a $(k-\ell)$\nobreakdash-defective $(k+\ell)$\nobreakdash-coloring. Moreover, there is an algorithm that constructs such a coloring in polynomial time.
\end{corollary}
Notice that \cref{cor:defective} is based on greedy coloring and therefore it is very robust.
For example, it holds also in more general settings of list colorings, paintability, DP\nobreakdash-colorings, and Alon-Tarsi numbers, see~\cite{LiLWZ2023} for details.

\section{Decomposition of type $(k,k-1)$}\label{sec:main}

This section is devoted to the proof of \cref{thm:main}, the most difficult part of \cref{thm:super}.
Our proof of \cref{thm:main} goes by induction on the size of a graph and shows a more general result given in \cref{lem:main}.
The main part of the inductive argument is in the proof of \cref{lem:friendly_high_rank}.
As every $1$\nobreakdash-arboric graph is a forest, it is $1$\nobreakdash-degenerate, and therefore admits a $(1,0)$\nobreakdash-decomposition. Consequently, \cref{thm:main} for $k=1$ holds  trivially. Thus, we fix an integer value $k \ge 2$ for the rest of this section.
The general idea of the proof is to give a procedure that, for any given graph~$G=(V,E)$, selects edges $E_2 \subseteq E$ and constructs a linear order~$\ll$ on~$V$ that is a $k$\nobreakdash-degenerate order of~$G\setminus E_2$.
We will say that an edge selected into $E_2$ is \emph{colored red}.
The procedure is carried out in steps.
In each step, some edge is colored red and removed, or some vertex~$v$ with degree at most $k$ in the current graph is placed last in $\ll$.
In the latter case, edges incident to~$v$ are \emph{colored blue} and removed from the graph together with the vertex $v$.
The procedure continues recursively in a smaller graph.
If we are able to ensure that every vertex is incident to at most $k-1$ red edges, then 
the blue edges span a $k$\nobreakdash-degenerate graph and the red edges span a $(k-1)$\nobreakdash-bounded-degree graph, certifying a $(k,k-1)$\nobreakdash-decomposition of $G$.

When there is a vertex with degree at most~$k$, the algorithm can place it last in the order without any trouble.
But, the arboricity condition gives only that there is always at least one vertex with degree at most~$2k-1$.
When the degree of~$v$ is between $k+1$ and~$2k-1$, and we want to place $v$ last in~$\ll$, we must color some edges incident to~$v$ red before doing so.
The main difficulty is that when we start coloring edges red, we must take it into account in the later steps of the procedure.
Before proceeding to the main result, we need to define some necessary notions.

In order to account for the red edges colored in earlier steps of the procedure, we say that each vertex has some \emph{external} red edges, possibly zero.
For the sake of our argument, it is enough to count these external edges for each vertex.
An \emph{object} $(G,\redc)$ consists of a graph $G=(V,E)$ and a \emph{red-count function} $\redc$ that assigns an integer $\redc(v) \in \set{0,1,\ldots,k-1}$ to each vertex $v \in V$.
For any subset $U \subseteq V$, we will use $\redc(U)$ to denote  $\sum_{u \in U}\redc(u)$.
We say that an object~$(G,\redc)$ is \emph{empty} when $G$ has no vertices, that $(G,\redc)$ is \emph{$k$\nobreakdash-arboric} when $G$ is $k$\nobreakdash-arboric, and that $(G,\redc)$ is \emph{internal} when there are no external edges, i.e., $\redc(V) = 0$.
For any subset of vertices~$U\subseteq V$, an object~$(G[U],\redc[U])$ with $\redc[U]$ being a restriction of $\redc$ to $U$ is called a \emph{subobject} of~$(G,\redc)$ \emph{spanned} by $U$, or simply a \emph{subobject}.
A \emph{$(k,k-1)$\nobreakdash-decomposition} of an object~$(G,\redc)$ is a partition of the edges of $G$ into two subsets $E = E_1 \sqcup E_2$ so that $G_1=(V,E_1)$ is $k$\nobreakdash-degenerate, and every vertex~$v$ has degree at most~$k-1-\redc(v)$ in~$G_2=(V,E_2)$.
We say that an object~$(G,\redc)$ \emph{decomposes} if it admits a $(k,k-1)$\nobreakdash-decomposition.
We usually use $\ll$ to denote a $k$\nobreakdash-degenerate order of~$G_1$.
Our goal is to prove \cref{thm:main}, which can now be reformulated as saying that every internal $k$\nobreakdash-arboric object decomposes.

In order to prove \cref{thm:main}, we need to define a slightly larger class of objects that is more useful for the induction.
For an object $(G,\redc)$ we use the \emph{potential function} $\pot_{G,\redc}$ defined for every subset $U \subseteq V$ of vertices as
\[
    \pot_{G,\redc}(U) = 2k\cdot\norm{U} - 2\cdot\norm{E[U]} - 2\redc(U) = \sum_{u \in U} \left(2k - \deg_{G[U]}(u) - 2\redc(u)\right)\text{.}
\]
We define the \emph{total potential} of $(G,\redc)$ to be $\Pot(G,\redc) = \pot_{G,\redc}(V)$.
An easy calculation shows that the potential function is submodular.
\begin{observation}\label{obs:potential_submodular}
For any subsets $U_1, U_2 \subseteq V$ of vertices in an object~$(G,\redc)$ we have 
\begin{align*}    
    \pot_{G,\redc}(U_1\cup U_2) + \pot_{G,\redc}(U_1\cap U_2) &= \pot_{G,\redc}(U_1) + \pot_{G,\redc}(U_2) - 2\cdot\norm{E[U_1\setminus U_2,U_2\setminus U_1]} \\
    &\le \pot_{G,\redc}(U_1) + \pot_{G,\redc}(U_2) \text{.}
\end{align*}
\end{observation}

Further, observe that the potential of any subset of vertices is always an even integer, and that the potential of the empty subset is $0$.
We say that an object~$(G,\redc)$ is \emph{positive} if, for every non-empty subset~$\emptyset \neq U \subseteq V$ of vertices, we have $\pot_{G,\redc}(U) \ge 2$.
By the Nash-Williams Theorem, we have the following bound on the potential of any non-empty subset in an internal $k$\nobreakdash-arboric object $(G,\redc)$:
\[
    \forall_{\emptyset \neq U \subseteq V}\,\pot_{G,\redc}(U) = 2k\cdot\norm{U} - 2\cdot\norm{E[U]} \ge 2k\cdot\norm{U} - 2\cdot\brac{k\cdot\brac{\norm{U}-1}} \ge 2k \ge 2\text{,}
\]
which implies the following.
\begin{observation}\label{obs:arboric_positive}
    Every internal $k$\nobreakdash-arboric object is positive.
\end{observation}

Objects with potential~$0$ include all internal $2k$\nobreakdash-regular graphs.
The statement of \cref{thm:main} is false for any $2k$\nobreakdash-regular graph, as it is not $(2k-1)$\nobreakdash-degenerate (an obvious necessary condition for a $(k,k-1)$\nobreakdash-decomposition). 
In fact, there is a wide variety of graphs with potential~$0$ for which the statement of \cref{thm:main} fails. This includes even $(k+1)$\nobreakdash-degenerate graphs, with the biclique $K_{k+1,k(k+1)}$ providing one such example. We therefore restrict our attention to positive objects. However, the statement of \cref{thm:main} is false even for some internal positive graphs. 

\begin{example}\label{exm:positive_no_decomposition}
Consider a graph $G=(V,E)$ with vertex set $V=\set{x_0,x_1,\dots x_{k+1},y_0,y_1,\dots,y_{k+1}}$ and $E$ having $2k-1$ copies of edge $\set{x_i,y_i}$ for each $i=1,\dots,k+1$, and single edges connecting vertex $x_0$ with every other vertex $x_i$ and $y_0$ with every other vertex $y_i$. See \cref{fig:obstructed_counterexample} for $k=2$. \newline
The graph $G$ is internal positive and it does not decompose.
\end{example}

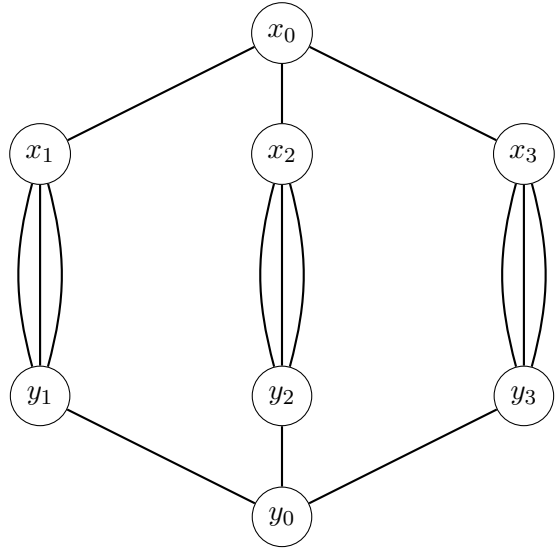
\begin{figure}[h]
	\centering
	\begin{tikzpicture}[scale=0.8]
	\node[draw, circle] (x0) at (0,4) {$x_0$};
	\node[draw, circle] (y0) at (0,-4) {$y_0$};
	\node[draw, circle] (x1) at (-4,2) {$x_1$};
	\node[draw, circle] (y1) at (-4,-2) {$y_1$};
	\node[draw, circle] (x2) at (0,2) {$x_2$};
	\node[draw, circle] (y2) at (0,-2) {$y_2$};
	\node[draw, circle] (x3) at (4,2) {$x_3$};
	\node[draw, circle] (y3) at (4,-2) {$y_3$};
	
	\draw[thick] (x0) -- (x1);
	\draw[thick] (x0) -- (x2);
	\draw[thick] (x0) -- (x3);
	\draw[thick] (y0) -- (y1);
	\draw[thick] (y0) -- (y2);
	\draw[thick] (y0) -- (y3);
	
	\foreach \i in {1,...,3} {
		\draw[thick] (x1) edge[bend left={-30+15*\i}] (y1);
		\draw[thick] (x2) edge[bend left={-30+15*\i}] (y2);
		\draw[thick] (x3) edge[bend left={-30+15*\i}] (y3);
	}
	\end{tikzpicture}
	\caption{An internal positive object that does not decompose for $k=2$.\label{fig:obstructed_counterexample}}
\end{figure}

It is an easy exercise to check that the internal object from \cref{exm:positive_no_decomposition} is indeed positive and that it does not decompose. The reason is the presence of what we refer to below as obstructions.

For any vertex $v$ of $G$, let the \emph{score} of~$v$ in~$(G,\redc)$ be defined as
\[
    \score_{G,\redc}(v) = \max(0,2k-\deg_G(v)-\redc(v))\text{.}
\]
We define the \emph{total score} of an object~$(G,\redc)$ to be $\score(G,\redc) = \sum_{v \in V}\score_{G,\redc}(v)$.
We say that a non-empty object~$(G,\redc)$ is an \emph{obstruction} when the total score of~$G$ is strictly smaller than $k + \frac{1}{2}\Pot(G,\redc)$.
We say that an object~$(G,\redc)$ is \emph{unobstructed} when every subobject of~$(G,\redc)$ is not an obstruction.

By \cref{obs:arboric_positive}, we have that any internal $k$\nobreakdash-arboric object is positive.
For the potential and the total score of an internal $k$\nobreakdash-arboric object, we get the following inequalities.
\begin{align*}
  \score(G,\redc)=\sum_{v\in V} \score_{G,\redc}(v) & \ge \sum_{v \in V} \brac{2k-\deg_G(v)} = 2k\norm{V} -2\norm{E} \\
    & = k(\norm{V}+1) - \norm{E} + k(\norm{V}-1) - \norm{E} \\ 
    & \ge  k + k\norm{V} - \norm{E} = k+\frac{1}{2}\Pot(G,\redc) 
    \text{.}
\end{align*}
As any subgraph of a $k$\nobreakdash-arboric graph is also $k$\nobreakdash-arboric, we obtain the following.
\begin{observation}\label{obs:arboric_friendly}
    Every internal $k$\nobreakdash-arboric object is unobstructed.
\end{observation}
Notice that, just by definitions, any subobject of a positive unobstructed object is a positive unobstructed object.
The remainder of the section builds a proof of \cref{lem:main} that states that every positive unobstructed object decomposes.
\cref{thm:main} will follow from \cref{lem:main} by \cref{obs:arboric_positive} and \cref{obs:arboric_friendly}.
Before presenting the proofs, we need to define some useful notions and introduce the necessary tools.

We say that a proper subset $\emptyset \neq X \subsetneq V$ of vertices is \emph{dangerous} in $(G,\redc)$ when $\pot_{G,\redc}(X) = 2$.
The reason we call such sets dangerous is that our algorithm, which is designed to keep the object positive, cannot increase the red-count of a vertex in a dangerous subset without removing some edges.
Notice that for any vertex $v$ with $\redc(v) = k-1$, we have that $\pot_{G,\redc}(\set{v}) = 2k - 2(k-1) = 2$ and hence the set $\set{v}$ is dangerous, except for a trivial case $V=\set{v}$.
We say that a vertex~$w$, a neighbor of vertex~$v$, is a \emph{safe neighbor} of~$v$ if for every dangerous subset~$X$, we have $w \in X \implies v \in X$.
Otherwise, $w$ is a \emph{dangerous neighbor} of~$v$.
Notice that even when the total potential of the object is $2$, we explicitly do not call set $V$ dangerous.

Consider any two dangerous subsets $X_1,X_2$ in a positive object $(G,\redc)$.
By \cref{obs:potential_submodular}, we get
\begin{align*}
    \pot_{G,\redc}(X_1\cup X_2)+\pot_{G,\redc}(X_1\cap X_2) &=\pot_{G,\redc}(X_1)+\pot_{G,\redc}(X_2)-2\norm{E[X_1\setminus X_2,X_2\setminus X_1]}\\
                            &=4-2\norm{E[X_1\setminus X_2,X_2\setminus X_1]}\text{.}
\end{align*}
As $(G,\redc)$ is positive, we have $\pot_{G,\redc}(X_1\cup X_2) \ge 2$ and $\pot_{G,\redc}(X_1\cap X_2) \ge 2$ when $X_1\cap X_2 \neq \emptyset$.
We easily reach the following observation.
\begin{observation}\label{obs:dangerous}
For any two dangerous subsets $X_1$, $X_2$ in a positive object~$(G,\redc)$, we have:
\begin{itemize}
    \item If $X_1 \cap X_2 = \emptyset$, then there is at most one edge in $E[X_1,X_2]$. Moreover, if there is an edge in $E[X_1,X_2]$ then $X_1\cup X_2$ is also a dangerous subset or the whole set $V$.
    \item If $X_1 \cap X_2 \neq \emptyset$, then both $X_1 \cup X_2$ and $X_1 \cap X_2$ are dangerous subsets or the whole set $V$.
    \item If $X_1 \cap X_2 \neq \emptyset$, then there are no edges in $E[X_1\setminus X_2,X_2\setminus X_1]$.
\end{itemize}
\end{observation}

In order to give an inductive proof that every positive unobstructed object decomposes, we use the following \emph{reductions} of three different \emph{types}.
Each reduction has a \emph{base} that is a set of at most two vertices and a positive integer \emph{rank} that will play a role later.
\begin{enumerate}
    \item\label{red:easy} Any vertex~$v$ of degree $\deg_G(v) \le k$ has $\score_{G,\redc}(v)\ge2k-k-(k-1)>0$ and is the base of a reduction of Type~\ref{red:easy} and rank equal to $\score_{G,\redc}(v)$.
        The algorithm applies this reduction by coloring all edges incident to $v$ blue, placing $v$ last in $\ll$, and removing $v$ from $G$.
    \item\label{red:hard} Any vertex~$v$ of degree $\deg_G(v) > k$ and of $\score_{G,\redc}(v) > 0$ with at least one safe neighbor $w$
        is the base of a reduction of Type~\ref{red:hard} and rank equal to $\score_{G,\redc}(v)$.
        The algorithm applies this reduction by coloring the edge $\set{v,w}$ red, removing the edge $\set{v,w}$ from $G$, and increasing the red-count of $w$ by one.
        Notice that the red-count of $v$ does not increase.
    \item\label{red:tree} Any edge $\set{v_1, v_2}$ in $E$ with
        $\deg_G(v_1) > k$, $\deg_G(v_2) > k$, 
        $\score_{G,\redc}(v_1) > 0$, $\score_{G,\redc}(v_2) > 0$,
        no safe neighbor of $v_1$, and no safe neighbor of $v_2$ is the base of a reduction of Type~\ref{red:tree} and rank equal to $\min(\score_{G,\redc}(v_1), \score_{G,\redc}(v_2))$.
        The algorithm applies this reduction by coloring the edge $\set{v_1,v_2}$ red and removing the edge $\set{v_1,v_2}$ from $G$.
        Notice that the red-counts of the vertices do not increase.
\end{enumerate}

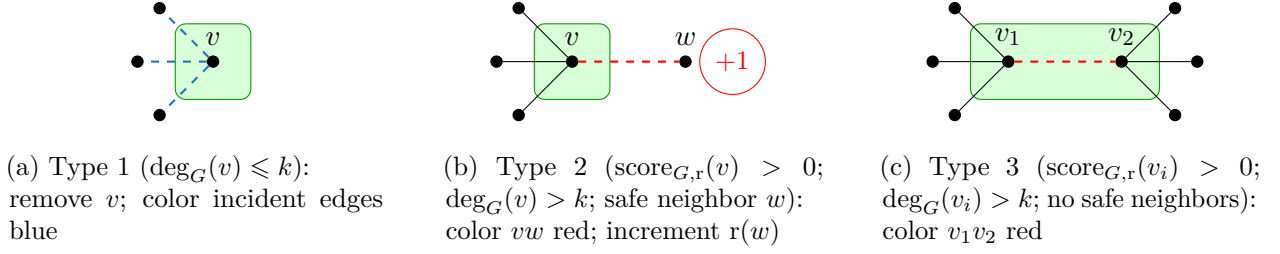
\begin{figure}[h]
  \centering
  \tikzset{
    vertex/.style={circle, draw, fill=black, inner sep=1.5pt},
    lbl/.style={font=\small},
    base/.style={draw=green!60!black, fill=green!15, rounded corners}
  }

  \begin{subfigure}[t]{0.3\textwidth}
    \centering
    \begin{tikzpicture}
      \filldraw[base] (-0.5, -0.5) rectangle (0.5, 0.5);
      \node[vertex, label=above:$v$] (v) at (0,0) {};
      \foreach \a in {135, 180, 225} {
        \node[vertex] (u\a) at (\a:1) {};
        \draw[defblue, dashed, thick] (v) -- (u\a);
      }
    \end{tikzpicture}
    \medskip
    \caption{Type~\ref{red:easy} ($\deg_G(v) \le k$):\\remove $v$; color incident edges blue}
  \end{subfigure}
  \hfill
  \begin{subfigure}[t]{0.3\textwidth}
    \centering
    \begin{tikzpicture}
      \filldraw[base] (-0.5, -0.5) rectangle (0.5, 0.5);
      \node[vertex, label=above:$v$] (v) at (0,0) {};
      \node[vertex, label=above:$w$] (w) at (1.5,0) {};
      \draw[red, dashed, thick] (v) -- (w);
      \node[lbl, red,  circle, draw, right=1mm] at (w.east) {$+1$};
      
      \foreach \a in {135, 180, 225} {
        \node[vertex] (u\a) at (\a:1) {};
        \draw (v) -- (u\a);
      }
    \end{tikzpicture}
    \medskip
    \caption{Type~\ref{red:hard} ($\score_{G,\redc}(v) > 0$; $\deg_G(v) > k$; safe neighbor $w$):\\ color $vw$ red; increment $\redc(w)$}
  \end{subfigure}
  \hfill
  \begin{subfigure}[t]{0.3\textwidth}
    \centering
    \begin{tikzpicture}
      \filldraw[base] (-0.5, -0.5) rectangle (2, 0.5);
      \node[vertex, label=above:$v_1$] (v1) at (0,0) {};
      \node[vertex, label=above:$v_2$] (v2) at (1.5,0) {};
      \draw[red, dashed, thick] (v1) -- (v2);
      
      \foreach \a in {135, 180, 225} {
        \node[vertex] (u1\a) at (\a:1) {};
        \draw (v1) -- (u1\a);
      }
      \foreach \a in {45, 0, -45} {
        \node[vertex] (u2\a) at ($(1.5,0) + (\a:1)$) {};
        \draw (v2) -- (u2\a);
      }
    \end{tikzpicture}
    \medskip
    \caption{Type~\ref{red:tree} ($\score_{G,\redc}(v_i) > 0$; $\deg_G(v_i) > k$; no safe neighbors):\\ color $v_1v_2$ red}
  \end{subfigure}
  \caption{Three types of reductions. Green rectangles indicate bases of reductions. Dashed lines denote colored edges that are removed from the graph~$G$.}
  \label{fig:reductions}
\end{figure}

First, observe that the red-count of a vertex is increased only in a reduction of Type~\ref{red:hard} and only for a safe neighbor of the base vertex.
We wanted the red-count to account for the red edges colored in earlier steps of the procedure, so this might be a little confusing at first, but it is deliberate.
The vague idea is that it does not make sense to further increase the red-count of a vertex once it has positive score. This gets clarified only in the proof of \cref{lem:main} when we do the final count of the total number of red edges incident to a vertex.

Second, observe that:
\begin{itemize}
\item each vertex in the base of a reduction of Type~\ref{red:easy} has degree at most $k$,
\item each vertex in the base of a reduction of Type~\ref{red:hard} has degree at least $k+1$, degree plus red-count at most $2k-1$, and a safe neighbor,
\item each vertex in the base of a reduction of Type~\ref{red:tree} has degree at least $k+1$, degree plus red-count at most $2k-1$, and no safe neighbor.
\end{itemize}
Therefore, only two different reductions of Type~\ref{red:tree} can have a non-empty intersection of the bases.
Notice, however, that if $\set{v_1, v_2}$ is the base of a reduction of Type~\ref{red:tree}, then there is only one edge between vertices $v_1$ and $v_2$ in $G$. As $v_1$ is a dangerous neighbor of $v_2$, there exists a dangerous subset $V_1$ that contains $v_1$ and does not contain $v_2$. Similarly, an analogous subset $V_2$ exists. By \cref{obs:dangerous} subsets $V_1,V_2$ have empty intersection and there is exactly one edge between them.

We say that a set of reductions is \emph{independent} if the bases of the reductions are pairwise disjoint.
We define the \emph{total rank} of an object $(G,\redc)$, denoted by $\rank(G,\redc)$, to be the maximum possible sum of the ranks of any independent set of reductions available in~$(G,\redc)$.
Observe that a reduction of Type~\ref{red:easy} or Type~\ref{red:hard} has a base consisting of a single vertex, and the rank of such a reduction equals the score of that vertex.
Reductions of Type~\ref{red:tree} have a base consisting of two vertices, and the rank of such a reduction equals the minimum score of the two vertices in the base.
Hence, the sum of the ranks of any independent set of reductions is at most the sum of the scores of the vertices in the bases of these reductions. Thus, the total rank of any object is at most the total score of this object.
Further, the total rank equals the total score if and only if every vertex with a positive score is the base of a reduction of Type~\ref{red:easy} or Type~\ref{red:hard}.

We say that an object~$(G',\redc')$ is \emph{smaller} than an object~$(G,\redc)$ if $G'$ has fewer vertices than $G$, or if both graphs have the same number of vertices and $G'$ has fewer edges than $G$.
Observe that applying any of the reductions either reduces the number of vertices or the number of edges and yields a smaller object.
The lemma shows that reductions preserve being positive and unobstructed.

\begin{lemma}\label{lem:friendly_reductions}
    Applying any of the reductions in a positive unobstructed object gives a smaller positive unobstructed object.
\end{lemma}
\begin{proof}
    Let $(G'=(V',E'),\redc')$ be the result of applying one of the reductions in a positive unobstructed object $(G=(V,E),\redc)$.
    If $(G',\redc')$ is a result of a reduction of Type~\ref{red:easy} or Type~\ref{red:tree}, then for every subset of vertices $U \subseteq V'$ we have $\pot_{G',\redc'}(U) \ge \pot_{G,\redc}(U) > 0$, so $(G',\redc')$ is positive.
    If $(G',\redc')$ is a result of a reduction of Type~\ref{red:easy}, then every subobject of $(G',\redc')$ is also a subobject of $(G,\redc)$, and we get that $(G',\redc')$ is unobstructed.

    Now, let $(G',\redc')$ be the result of a reduction of Type~\ref{red:tree} and base $\set{v_1,v_2}$.
    We need to show that $(G',\redc')$ is unobstructed.
    Assuming to the contrary, let the subobject spanned by $U \subseteq V'$ be an obstruction.
    We have $v_1,v_2 \in U$, as otherwise the subobject spanned by $U$ in $(G',\redc')$ would be the same as the subobject spanned by $U$ in $(G,\redc)$.
    Thus, $\redc'[U] = \redc[U]$, while $G'[U]$ and $G[U]$ differ by exactly one edge $\set{v_1,v_2}$.
    The total potential of $(G'[U],\redc'[U])$ is higher by two than the total potential of $(G[U],\redc[U])$.
    The scores of vertices $v_1$ and $v_2$ in $(G'[U],\redc'[U])$ are one higher than in $(G[U],\redc[U])$.
    The scores of other vertices in $U\setminus\set{v_1,v_2}$ are the same in both subobjects.
 As the subobject $(G[U],\redc[U])$ spanned by $U$ in $(G,\redc)$ is not an obstruction, we get:
    \begin{align*}
        \score(G'[U],\redc'[U]) = \score(G[U],\redc[U])+2 \ge k + \frac{1}{2}\Pot(G[U],\redc[U])+2 = k + \frac{1}{2}\Pot(G'[U],\redc'[U])+1\text{,}
    \end{align*}
    and $(G'[U],\redc'[U])$ is not an obstruction -- a contradiction.

    Now, for the remaining case, let $(G',\redc')$ be the result of applying a reduction of Type~\ref{red:hard} that removes from $G$ an edge $\set{v,w}$ between a vertex~$v$ and a safe neighbor $w$ of $v$, and increases the red-count of $w$.
    Notice that we have $\redc(w)\le k-2$, as otherwise the subset $\set{w}$ would be dangerous and $w$ would not be a safe neighbor of $v$.
    The score of~$v$ is positive in~$(G,\redc)$.
    The reduction increases the score of $v$ by one and does not change the score of $w$ or any other vertex.
    The potential of any subset that includes $w$ but does not include $v$ decreases by $2$.
    As $w$ is a safe neighbor of $v$ in $(G,\redc)$, the potential of any such set is at least $4$ in $(G,\redc)$ and remains positive in $(G',\redc')$.
    The potential of any other subset does not change, so $(G',\redc')$ is positive.
    It remains to show that $(G',\redc')$ is unobstructed.
    Assuming to the contrary, let the subobject spanned by $U \subseteq V'$ be an obstruction.
    The score of $x$ in $G'[U]$ is at least the same as in $G[U]$ except possibly when $U\ni x=w$ and $v \notin U$. In this case, the potential of $U$ in $(G,\redc)$ is $2$ greater than in $(G',\redc')$. Thus, we would get that $(G[U],\redc[U])$ is also an obstruction -- a contradiction.
    Otherwise, when $w \notin U$ or $v \in U$, we have $\score(G'[U],\redc'[U]) \ge \score(G[U],\redc[U])$ and the potential of $U$ in $(G',\redc')$ is the same as in $(G,\redc)$, so we would again get that $(G[U],\redc[U])$ is also an obstruction -- a contradiction.
    
    Thus, in each possible case, we have that $(G',\redc')$ is positive and unobstructed.
\end{proof}

Having defined all the necessary notions, let us briefly return to the class of objects we consider. 
Positive and unobstructed objects form a natural class that contains all internal $k$-arboric objects, is closed under the reductions, and is well suited for induction -- since, as we will show, every such object reduces to a smaller one. Moreover, in a certain sense, this is the largest class with these properties. Indeed, as the next example shows, the presence of even a single obstruction in a class makes it possible to construct positive objects that do not admit any reduction.

\begin{example}\label{exm:obstruction_no_reduction}
Let $(G=(V,E),\redc)$ be a positive object of total potential $2$, which is an obstruction. 
Let $v_1,\dots,v_k$ be a list of $k$ vertices such that any $v\in V$ appears on the list at least $\score_{G,\redc}(v)$ times.
Consider an object $(G',\redc')$ constructed from $k+1$ copies of $(G,\redc)$ and additional vertices $w_1,\dots,w_k$ by joining with a single edge each $w_i$ to $v_i$ in all copies of $(G,\redc)$. Then,
the object $(G',\redc')$ is positive (with total potential $2$) obstructed (by $(G,\redc)$), but its total rank is $0$.
\end{example}

Therefore, the unobstructed condition should be viewed not as a merely technical assumption, but as an intrinsic requirement of our induction scheme. It ensures that every subobject has sufficient capacity for the reduction process.

We now turn to the proof. We first show that if a positive object $(G,\redc)$ has a dangerous subset $A$ whose neighbors are all safe, then every external edge from $A$ in $G$ decreases, by at most one, the rank in $G$ of an independent set of reductions of the subobject spanned by $A$. It is a natural bound, as every external edge from $A$ changes the score of at most one vertex in $A$ by decreasing it by one.

\begin{lemma}\label{lem:safe_stabs}
    Let $(G=(V,E),\redc)$ be a positive object and let $A\subset V$ be a dangerous subset of $(G,\redc)$.\newline
    Let $\rho = \rank(G[A],\redc[A])$ be the total rank of the subobject spanned by $A$.
    If for every edge $\set{a,w} \in E$ with $a \in A$ and $w \notin A$ we have that $w$ is a safe neighbor of $a$ in $(G,\redc)$, then
    there exists an independent set $\Psi$ of reductions in $(G,\redc)$ with the base of every reduction in $\Psi$ contained in $A$ and the total rank of reductions in $\Psi$ equal at least $\rho - \norm{E[A,V\setminus A]}$.
\end{lemma}
\begin{proof}
    Observe that if a vertex $a \in A$ has score $s$ in $(G[A],\redc[A])$ and has $m$ neighbors in $V \setminus A$, then the score of $a$ in $(G,\redc)$ is $\max(0,s-m)$.
    Thus, the total score of vertices in the bases of reductions in $(G[A],\redc[A])$ can decrease by at most $\norm{E[A,V\setminus A]}$ when we consider the whole object $(G,\redc)$ instead of the subobject $(G[A],\redc[A])$. However, one has to check that vertices with this total score indeed lead to the corresponding reductions. The intuition is that the safe neighbors in $V\setminus A$ decrease the score, but also allow to easily construct reductions of Type~\ref{red:hard}.

    Let $\Psi = \set{\psi_1,\psi_2,\ldots,\psi_j}$ be an independent set of reductions of total rank $\rho$ in $(G[A],\redc[A])$.
    For any $i=1,2,\ldots,j$, let $\rho_i$ be the rank of the reduction $\psi_i$ in $(G[A],\redc[A])$.
    Let $m_i$ be the number of edges in $E$ with one vertex in $V \setminus A$, and the second end in the base of $\psi_i$.
    Assume $\rho_i > m_i$.
    Then, the score of any vertex $v$ in the base of $\psi_i$ is at least $\rho_i$ in $(G[A],\redc[A])$ and at least $\rho_i-m_i > 0$ in~$(G,\redc)$.
    Observe that for any edge $\set{a,v}$ with $a \in A$, if $v$ is a dangerous neighbor of $a$ in $(G,\redc)$, then $v \in A$ (by the lemma's assumptions) and there exists a dangerous subset $W \subsetneq V$ with $v \in W$ and $a \notin W$.
    As $A$ is dangerous, by \cref{obs:dangerous}, $W\cap A$ is also a dangerous subset in $(G[A],\redc[A])$ and thus $v$ is a dangerous neighbor of $a$ in $(G[A],\redc[A])$ as well.
    Therefore, if $m_i=0$, then the reduction $\psi_i$ is also a reduction of the same type, base, and rank in $(G,\redc)$.
    If $m_i\ge 1$ and $\psi_i$ is a reduction of Type~\ref{red:tree}, then let $a$ be one of the vertices in the base of $\psi_i$ with at least one neighbor in~$V \setminus A$.
    We get that $a$ has a safe neighbor in $(G,\redc)$ and is the base of a reduction of Type~\ref{red:hard} and the score of $a$ is at least $\rho_i-m_i$ in $(G,\redc)$.
    If $m_i\ge 1$ and $\psi_i$ is of Type~\ref{red:easy} or Type~\ref{red:hard}, then let $a$ be the vertex in the base of $\psi_i$.
    We have that $a$ has positive score and a safe neighbor in $(G,\redc)$.
    Thus, vertex $a$ is the base of a reduction of Type~\ref{red:easy} or Type~\ref{red:hard} and the score of $a$ is $\rho_i-m_i$ in $(G,\redc)$.

    We get that for every $i=1,2,\ldots,j$ with $\rho_i > m_i$, there is a reduction in $(G,\redc)$ with the base contained in the base of $\psi_i$ and rank at least $\rho_i - m_i$.
    As $\sum_{i=1}^{j} \rho_i = \rho$ and $\sum_{i=1}^{j} m_i \le \norm{E[A,V\setminus A]}$,
    the statement of the lemma follows by summing up ranks of all reductions obtained this way for $i=1,2,\ldots,j$.
\end{proof}

For the proof of \cref{lem:main} it would suffice to show that every non-empty positive unobstructed object has at least one reduction. However, to make the induction work, we require the following stronger statement.

\begin{lemma}\label{lem:friendly_high_rank}
    Let~$(G,\redc)$ be a non-empty positive unobstructed object, then
    \[
        \rank(G,\redc) \ge k+\frac{1}{2}\Pot(G,\redc)\text{.}
    \]
\end{lemma}
\begin{proof}
    The proof is by induction on the ordering induced by the smaller than relation.
    For the base of the induction, consider an object $(G,\redc)$ with just one vertex $v$.
    There are no edges in $G$.
    Thus,
    \[
        \Pot(G,\redc) = \pot_{G,\redc}(\set{v})=2k-2\redc(v)\text{.}
    \]
    There is just one reduction of Type~\ref{red:easy} available in~$(G,\redc)$.
    The rank of this reduction is given by $\score_{G,\redc}(v) = 2k-\redc(v) = k+\frac{1}{2}\Pot(G,\redc)$.
    This concludes the proof for the base case.

    Now, let $(G=(V,E),\redc)$ be a positive unobstructed object with $\norm{V} \ge 2$. By the induction hypothesis, we have that every positive unobstructed object smaller than $(G,\redc)$ satisfies \cref{lem:friendly_high_rank}.
    We divide the proof into a series of claims, depending on the structure of the object~$(G,\redc)$.

    \begin{claim}\label{clm:connected}
        If $G$ is disconnected, then \cref{lem:friendly_high_rank} follows.
    \end{claim}
    \begin{claimproof}
        Let $X$, $Y$ be two non-empty disjoint subsets of vertices such that $V=X\sqcup Y$ and $E[X,Y]=\emptyset$.
        By the induction hypothesis applied to the subobjects spanned by $X$ and by $Y$, we get 
     \begin{align*}
            \rank(G,\redc) &= \rank(G[X],\redc[X])+\rank(G[Y],\redc[Y]) \geq \brac{k+\frac{1}{2}\pot_{G,\redc}(X)} + \brac{k+\frac{1}{2}\pot_{G,\redc}(Y)} \\ 
            &= 2k + \frac{1}{2}\pot_{G,\redc}(V) \geq k + \frac{1}{2}\pot_{G,\redc}(V)\text{.}
     \end{align*}
    \end{claimproof}
    For the rest of the proof, we can assume that $G$ is connected.
    In particular, the degree of each vertex is at least $1$.

    \begin{claim}\label{clm:bridge}
       If the vertex set $V$ can be partitioned into two disjoint dangerous subsets $X$ and $Y$, then \cref{lem:friendly_high_rank} follows.
    \end{claim}
    \begin{claimproof}
    	If there is no edge between $X$ and $Y$, then the claim follows from \cref{clm:connected}. Otherwise, there is at least one edge in $E[X,Y]$. And, by \cref{obs:dangerous}, there is at most one edge in $E[X,Y]$.
        Let $\set{x,y}$ be the only edge with $x\in X$ and $y\in Y$.
        Let $(G',\redc')$ and $(G'',\redc'')$ be the subobjects of~$(G,\redc)$ spanned by $X$ and by $Y$, respectively.
        We have $\Pot(G,\redc) = \Pot(G',\redc') = \Pot(G'',\redc'') = 2$.
        By the induction hypothesis, the total rank of $(G',\redc')$ and the total rank of $(G'',\redc'')$ are at least $k+1$ each.
        Let $\Psi'=\set{\psi'_1,\psi'_2,\ldots,\psi'_i}$ be an independent set of reductions in $(G',\redc')$ of total rank at least $k+1$.
        Similarly, let $\Psi''=\set{\psi''_1,\psi''_2,\ldots,\psi''_j}$ be an independent set of reductions in $(G'',\redc'')$ of total rank at least $k+1$.

        For a vertex $v \in X$, $v \neq x$, let $w$, if any, be a dangerous neighbor of $v$ in $(G,\redc)$.
        As $v \neq x$, we have $w \in X$.
        By the definition, there exists a dangerous subset $W$ in $(G,\redc)$ with $w\in W$, and $v \notin W$.
        By \cref{obs:dangerous}, the set $W\cap X$ is also dangerous in $(G,\redc)$ and in $(G',\redc')$.
        Thus, $w$ is a dangerous neighbor of $v$ in $(G',\redc')$.
        Similarly, for a vertex $v \in Y$, $v \neq y$ every neighbor $w$ of $v$ that is dangerous in $(G,\redc)$ is also dangerous in $(G'',\redc'')$.
        As a result, we get that every reduction $\psi'$ in $\Psi'$ with $x$ not in the base of $\psi'$ and every reduction $\psi''$ in $\Psi''$ with $y$ not in the base of $\psi''$ is also a reduction of the same type, base, and rank in $(G,\redc)$.

        Let $\rho'$ be the rank of the reduction $\psi'_x$ in $\Psi'$ with $x$ in the base of $\psi'_x$, or $\rho'=0$ if there is no such reduction.
        Let $\rho''$ be the rank of the reduction $\psi''_y$ in $\Psi''$ with $y$ in the base of $\psi''_y$, or $\rho''=0$ if there is no such reduction.
        We get that the total rank of reductions in $(G,\redc)$ with neither $x$ nor $y$ in the base is at least $2k+2-\rho'-\rho''$.
        We consider the case $\rho'+\rho''\geq k+2$, as otherwise the claim immediately follows.
        Without loss of generality, we may also assume that $\rho' \le \rho''$.
        
        If $\rho' > k$, then $\deg_{G'}(x) < k$ and $\deg_{G}(x) \le k$, and hence $x$ is the base of a reduction of Type~\ref{red:easy} in~$(G,\redc)$ with rank $\rho'-1 \ge k$.
        As $\rho'' \ge \rho' > k$, similarly $y$ is the base of a reduction of Type~\ref{red:easy} in $(G,\redc)$ with rank at least $k$.
        In this case, we get that the total rank of $(G,\redc)$ is at least $2k \ge k+1$.

        If $\rho' \le k$, then there is at least one reduction $\psi'$ in $\Psi'$ with the base of $\psi'$ contained in $X\setminus\set{x}$.
        If additionally $\rho'' > k$, then again $y$ is the base of a reduction of Type~\ref{red:easy} in $(G,\redc)$ with rank at least~$k$.
        Together with the reduction $\psi'$ we get that the total rank of $(G,\redc)$ is at least $k+1$.

        It remains to check the case $\rho' \le \rho'' \le k$, whence $\rho'-1\geq 1$, since $\rho'+\rho''\geq k+2$.
        In this case the total rank of reductions in $(G,\redc)$ with neither $x$ nor $y$ in the base is at least $2k+2 -\rho' - \rho'' \ge k+2 - \rho' = k+1 - (\rho'-1)$, so it remains to find one additional reduction of rank at least $\rho'-1$.
        Observe that $\score_{G,\redc}(x) \ge \rho'-1$ and $\score_{G,\redc}(y) \ge \rho''-1 \ge \rho'-1$.
        If $x$ or $y$ is the base of a reduction of Type~\ref{red:easy} or Type~\ref{red:hard} in $(G,\redc)$, then we are done.
        Otherwise, the edge $\set{x,y}$ is the base of a reduction of Type~\ref{red:tree} and rank at least $\rho'-1$.
    \end{claimproof}

    \begin{claim}\label{clm:dangerous}
	If there is no dangerous subset, then \cref{lem:friendly_high_rank} follows.
\end{claim}
\begin{claimproof}
	If there is no dangerous subset, then every neighbor of every vertex is a safe neighbor.
	Thus, each vertex~$v$ with positive score is a base of a reduction of Type~\ref{red:easy} or Type~\ref{red:hard} of rank equal to the score of $v$.
	Thus, the total rank of~$(G,\redc)$ equals the total score of~$(G,\redc)$.
	As $(G,\redc)$ is not an obstruction, we get that the total score, and the total rank, is at least $k + \frac{1}{2}\Pot(G,\redc)$.
\end{claimproof}

For the main part of the proof we assume that none of the previous claims applies. We construct a maximal sequence $A_1,A_2,\ldots,A_p$ of pairwise disjoint dangerous subsets of $V$ satisfying a condition that $A_{i+1}$ is inclusion maximal in $V\setminus \bigcup_{j=1}^{i}A_j$.

Let $A_1$ be some inclusion maximal dangerous subset in $(G,\redc)$.
The existence of $A_1$ is guaranteed, as \cref{clm:dangerous} does not apply.
If there are no dangerous subsets contained in $V\setminus A_1$, then $p=1$ and the construction of the sequence is finished. 

Otherwise, observe that for any dangerous subset $Y\subset V\setminus A_1$ there is no edge between $A_1$ and $Y$. Indeed, if there is an edge in $E[A_1,Y]$, then by \cref{obs:dangerous} there is exactly one edge in $E[A_1,Y]$ and $\pot_{G,\redc}(A_1\cup Y) = 2$. If $A_1\cup Y\subsetneq V$, then it is also dangerous and contradicts maximality of $A_1$. Thus, $A_1\cup Y = V$. This is also impossible, as \cref{clm:bridge} does not apply. Hence, every dangerous subset $Y$ contained in $V\setminus A_1$ satisfies $E[A_1,Y] = \emptyset$. Let $A_2$ be some inclusion maximal dangerous subset contained in $V\setminus A_1$. We have  $E[A_1,A_2]=\emptyset$.

Similarly, for any $i \ge 2$, let $A_{i+1}$ be some inclusion maximal dangerous subset in $V\setminus \bigcup_{j=1}^{i}A_j$.\newline If such a set does not exist, then $p=i$ and the construction of the sequence is finished. Otherwise, $A_{i+1}$ is a dangerous subset contained in~$V\setminus A_1$ and, as observed before, we have $E[A_1,A_{i+1}] = \emptyset$. Furthermore, for every $2\le j \le i$, we have $E[A_j,A_{i+1}] = \emptyset$, as otherwise the set $A_j \cup A_{i+1}$ would be a dangerous superset of $A_j$ contained in $V\setminus\bigcup_{s=1}^{j-1}A_{s}$, contradicting maximality of $A_j$.

The construction of the sequence eventually finishes for some $p\ge 1$ and gives a non-empty sequence $A_1,A_2,\ldots,A_p$ of pairwise disjoint dangerous subsets of $V$.
For all $1 \le i < j \le p$, we have $E[A_i,A_j]=\emptyset$.
Let $A = \bigcup_{i=1}^{p}A_i$.
If $p=1$, then $A=A_1\neq V$.
Otherwise, for $p\ge2$, sets in the sequence span a disconnected graph and, as \cref{clm:connected} does not apply, we also have $A\neq V$.

The set $A$ has a very important property -- every dangerous neighbor in $(G,\redc)$ is in $A$.
Indeed, let $v$ be any vertex in $V$ and $w$ be a dangerous neighbor of $v$.
Then, there exists a dangerous set $W$ with $w \in W$ and $v \notin W$.
The construction guarantees that $V\setminus A$ does not contain any dangerous subset, and thus $W\cap A \neq \emptyset$.
Let $A_i$ be the first set in the sequence with $W\cap A_i \neq \emptyset$.
If $i=1$, as $W\cup A_1$ has potential $2$, we have that $W \subseteq A_1$ or $W\cup A_1 = V$.
In the first case, we get $w \in A_1 \subseteq A$.
In the second case, as we have $A_1 \cap W \neq \emptyset$, then by \cref{obs:dangerous}, there are no edges in $E[A_1\setminus W,W\setminus A_1]$.
As $v \notin W$ and $\set{v,w}\in E$, we get $v \in A_1$ and $w \in A_1 \subseteq A$.
If, in turn, $i>1$, then by \cref{obs:dangerous}, $A_i \cup W$ is dangerous and disjoint from $\bigcup_{j=1}^{i-1}A_j$.
As $A_i$ is an inclusion maximal set with these properties, we get $W \subseteq A_i$, and therefore $w \in A_i \subset A$.

We define the following partition of the vertices remaining in $V\setminus A$:
    \begin{itemize}
        \item $B$ is the set of vertices in $V\setminus A$ with score at least $1$ that ARE NOT the base of a reduction of Type~\ref{red:easy} or Type~\ref{red:hard},
        \item $C$ is the set of vertices in $V\setminus A$ with score at least $1$ that ARE the base of a reduction of Type~\ref{red:easy} or Type~\ref{red:hard},
        \item $D$ is the set of vertices in $V\setminus A$ with score equal to $0$.
    \end{itemize}
For any two sets $U$, $W$ among the sets $A$, $B$, $C$, $D$ we use $UW$ to denote the set of edges $E[U,W]$.
For a vertex $b \in B$ we know that $b$ is not the base of a reduction of Type~\ref{red:easy}, hence $\deg_G(b) > k$. Moreover, $b$ is also not the base of a reduction of Type~\ref{red:hard}, so every neighbor of $b$ is dangerous. Therefore, every neighbor of $b$ is in $A$. Summing over all vertices in $B$ we obtain:
    \begin{align}\label{eq:AB}
        \norm{BB} = 0, \quad \norm{BC} = 0, \quad \norm{BD} = 0, \quad \norm{AB} \ge (k+1)\norm{B}\text{.}
    \end{align}
For a vertex $d \in D$ we have $\score_{G,\redc}(d) = 0$.
Thus, $2k - \deg_G(d) - \redc(d) \le 0$, and $\deg_G(d) + \redc(d) \ge 2k$.
Summing this inequality over all vertices $d\in D$ and using~(\ref{eq:AB}) for $\norm{BD} = 0$ we obtain:
    \begin{align}
        \label{eq:D}
      \norm{AD} + \norm{CD} + 2\norm{DD} + \redc(D) \ge 2k\norm{D}\text{.}
    \end{align}
Since $A$ is the union of $p$ dangerous subsets with no edges between them, we get that $\pot_{G,\redc}(A) = 2p$.
Now, the calculation of the potential of the set $A\cup B$ and inequality~(\ref{eq:AB}) give the following estimate:
    \begin{align}\label{eq:B}
      2 &\le \pot_{G,\redc}(A\cup B) = \pot_{G,\redc}(A) + \pot_{G,\redc}(B) - 2\norm{AB} \nonumber\\
      &= 2p + 2k\norm{B} - 2\redc(B) - 2\norm{BB} -2\norm{AB} \nonumber\\
      &\le 2p + 2k\norm{B} -2\norm{AB} \nonumber \nonumber\\ 
      &\le 2\bbrac{p + k\norm{B} - \norm{AB}} + 2\bbrac{\norm{AB} - (k+1)\norm{B}} 
      = 2\bbrac{p-\norm{B}}\text{.}
    \end{align}
Next, the calculation of the total potential using~(\ref{eq:AB}) and inequalities~(\ref{eq:D})~and~(\ref{eq:B}) gives
    \begin{align}\label{eq:total}
      2 &\le \Pot(G,\redc) = \pot_{G,\redc}(A\cup B) + \pot_{G,\redc}(C \cup D) - 2\norm{AC} - 2\norm{AD} \nonumber\\
      &\le 2\bbrac{p-\norm{B}+k\norm{C}-\redc(C)-\norm{CC}+k\norm{D}-\redc(D)-\norm{DD}- \norm{CD}-\norm{AC}-\norm{AD}} \nonumber\\
      &\quad + 2\bbrac{\redc(D)+\norm{AD}+\norm{CD}+2\norm{DD}-2k\norm{D}} \nonumber\\
      &= 2\bbrac{p-\norm{B}+k\norm{C}-\redc(C)-\norm{CC}- k\norm{D}+\norm{DD}-\norm{AC}}\text{.}
    \end{align}

By the induction hypothesis, for each $i=1,\ldots,p$ we get that $\rank(G[A_i],\redc[A_i]) \ge k+\frac{1}{2}\cdot 2=k+1$.
Since each edge that connects $A_i$ with the rest of the graph satisfies the conditions of \cref{lem:safe_stabs}, we conclude that there is an independent set of reductions in~$(G,\redc)$ of total rank $k+1 - \norm{E[A_i,V\setminus A_i]}$ with the bases contained in $A_i$.
Further, for each vertex $c \in C$, there is a reduction of Type~\ref{red:easy} or Type~\ref{red:hard} with the base $c$ and rank equal to $\score_{G,\redc}(c)$.
Combining these two observations with equations~(\ref{eq:AB}) and inequalities~(\ref{eq:B})~and~(\ref{eq:total}) we get the following estimate for the total rank of~$(G,\redc)$:
\begin{align*}
  \rank(G,\redc) &\ge (k+1)p - \norm{AB} - \norm{AC} - \norm{AD} + \score_{G,\redc}(C)\\ 
    & = (k+1)p - \norm{AB} - \norm{AC} - \norm{AD} + 2k\norm{C} - \redc(C) - \norm{AC} - 2\norm{CC} - \norm{CD}\\
    & = \frac{1}{2}{\Pot(G,\redc)}+kp-k\norm{B}+k\norm{C}-k\norm{D}+\redc(B)+\redc(D)-\norm{AC}-\norm{CC}+\norm{DD}\\
    & \ge \frac{1}{2}{\Pot(G,\redc)}+(k-1)(p-\norm{B})
    +\Big(p-\norm{B}+k\norm{C}-k\norm{D}-\redc(C)-\norm{AC}-\norm{CC}+\norm{DD}\Big)\\
    & \ge \frac{1}{2}{\Pot(G,\redc)}+(k-1)+1\text{,}
\end{align*}
and hence, the object $(G,\redc)$ satisfies \cref{lem:friendly_high_rank}.
\end{proof}

\begin{lemma}\label{lem:main}
    Every positive unobstructed object decomposes.
\end{lemma}
\begin{proof}
    Let $(G,\redc)$ be a positive unobstructed object.
    By \cref{lem:friendly_high_rank} we have that every non-empty positive unobstructed object has at least one reduction.
    By \cref{lem:friendly_reductions} we have that applying that reduction gives a smaller positive unobstructed object.
    We select a sequence of reductions that starts from the object $(G,\redc)$ and results in an empty object with an additional property that whenever we have a choice, we prefer to use a reduction of Type~\ref{red:easy} rather than a reduction of a different type.
    Let $(H_1,\redc_1), (H_2,\redc_2),\ldots,(H_s,\redc_s)$ be a sequence of objects such that $(H_1,\redc_1)=(G,\redc)$, and $(H_s,\redc_s)$ is the empty object.
    Further, for every $i=1,2,\ldots,s-1$, there is a reduction $\psi_i$ in $(H_{i},\redc_{i})$ with $(H_{i+1},\redc_{i+1})$ being the result of applying $\psi_i$ in $(H_i,\redc_i)$.
    As mentioned before, each reduction $\psi_i$ is not of Type~\ref{red:hard} or Type~\ref{red:tree} if there is a different reduction of Type~\ref{red:easy} available in $(H_i,\redc_i)$.

    Let~$\ll$ be the order on vertices given by the sequence of reductions and~$E_2$ be the set of edges colored red by reductions of Type~\ref{red:hard} and of Type~\ref{red:tree}.
    The definition of reductions guarantees that $\ll$ is a $k$\nobreakdash-degenerate order of $G\setminus E_2$.

    It remains to prove that for every vertex $v$ the degree of $v$ in $G_2=(V,E_2)$ is at most $k-1-\redc_1(v)$.
    If $v$ is not in the base of any reduction of Type~\ref{red:hard} or Type~\ref{red:tree} in the sequence and $v$ is removed from the graph $H_i$ in the sequence by a reduction $\psi_i$ of Type~\ref{red:easy}, then $v$ is incident to exactly $\redc_i(v)-\redc_1(v) \le k-1-\redc_1(v)$ red edges.
    Otherwise, let $i$ be the first index in the sequence where $v$ is in the base of a reduction $\psi_i$ of Type~\ref{red:hard} or Type~\ref{red:tree}.
    We have that the score of $v$ in $(H_i,\redc_i)$ is positive and equal to $2k-\redc_i(v)-\deg_{H_i}(v)$.
    Thus, $\redc_i(v)+\deg_{H_i}(v) \le 2k-1$.
    As $i$ is the first index where reduction of Type~\ref{red:hard} or Type~\ref{red:tree} is applied to $v$, we have that $\redc_i(v)-\redc_1(v)$ equals the number of edges incident to~$v$ colored red by reductions $\psi_1,\psi_2,\ldots,\psi_{i-1}$.
    Each time another edge incident to $v$ is colored red by reductions $\psi_i,\psi_{i+1},\ldots,\psi_{s-1}$, the degree of $v$ is decreased by at least one.
    If at some point the degree of $v$ drops to $k$, then $v$ is the base of a reduction of Type~\ref{red:easy}.
    The construction of the reduction sequence guarantees that from this point, no more edges incident to $v$ are colored red.
    Thus,
    \[
        \deg_{G_2}(v) \le \redc_i(v)-\redc_1(v)+(\deg_{H_i}(v) - k) \le 2k - 1 - \redc_1(v) - k = k-1-\redc_1(v)\text{.}
    \]
    We have shown that every vertex $v$ in the graph $G_2$ has degree at most $k-1-\redc_1(v)$, and the proof is finished.
\end{proof}

We are ready to recall and prove  \cref{thm:main}.
\thmmain*
\begin{proof}
Given a $k$\nobreakdash-arboric graph $G$, we define red-count function $\redc$ that assigns value $0$ to every vertex of $G$.
The object~$(G,\redc)$ is internal and $k$\nobreakdash-arboric.
By \cref{obs:arboric_positive} and \cref{obs:arboric_friendly} we have that $(G,\redc)$ is positive and unobstructed.
By \cref{lem:main} we have that $(G,\redc)$ decomposes.
The $(k,k-1)$\nobreakdash-decomposition of $(G,\redc)$ gives a $(k,k-1)$\nobreakdash-decomposition of $G$.
\end{proof}

\section{Decomposition type changes}\label{sec:super}

In this section we show that \cref{thm:main} can be generalized to the sufficient part of \cref{thm:super}. The following lemma allows to modify a $(d,h)$\nobreakdash-decomposition of a graph into a $(d+1,h-1)$\nobreakdash-decomposition. As a consequence, it allows to prove \cref{thm:super} by an application of \cref{thm:main} and repeated applications of \cref{lem:super}.

\lemsuper*
\begin{proof}
Let $G=(V,E)$ be a graph with a partition of the edge set $E = E_1 \sqcup E_2$ such that $G_1=(V,E_1)$ is $d$\nobreakdash-degenerate and $G_2=(V,E_2)$ is $h$\nobreakdash-bounded-degree.
Let $\ll$ be a $d$\nobreakdash-degenerate order of $G_1$.
As before, we say that the edges in $E_1$ are colored \emph{blue} and the edges in $E_2$ are colored \emph{red}.
	
We show a procedure that selects some red edges and colors them \emph{violet}.
Then we define $E_1'$ to be the set of blue and violet edges and $E_2'$ to be the set of red edges that are not violet.
Let $G_1'=(V,E_1')$ and $G_2'=(V,E_2')$ be the graphs induced by the new edge partition.
We prove that it is a $(d+1,h-1)$\nobreakdash-decomposition.
Moreover, the order $\ll$ which is a $d$\nobreakdash-degenerate order of $G_1$ is also a $(d+1)$\nobreakdash-degenerate order for $G_1'$.
The construction ensures that for every vertex $v \in V$ we have:
\begin{itemize}
\item Vertex~$v$ is incident to at most one violet edge that joins $v$ to a vertex that is before $v$ in $\ll$.
\item If $v$ has $h$ incident red edges in $E_2$, then $v$ is incident to at least one violet edge.
\end{itemize}
It is clear that these two conditions guarantee that the resulting graph $G_1'$ is $(d+1)$\nobreakdash-degenerate and that $G_2'$ is $(h-1)$\nobreakdash-bounded-degree.
	
We say that a vertex $v$ in $V$ is \emph{heavy} if $v$ has $h$ incident red edges and all these edges join $v$ to vertices that are after $v$ in $\ll$.
Every other vertex is \emph{light}.
Observe that no two heavy vertices are connected by a red edge.
Let $H$ be a bipartite graph spanned by the red edges incident to heavy vertices.
Every heavy vertex has degree $h$ in $H$, and every other vertex has degree at most $h$ in $H$ as red edges span an $h$\nobreakdash-bounded-degree graph.
By Hall's theorem, there is a matching $M$ in $H$ that saturates all heavy vertices.
	
We select the violet edges in two steps. 
First, we select all edges in the matching~$M$ to be violet.
This way we ensure that every heavy vertex is incident to a violet edge.
Next, for every light vertex $v$ that is incident to $h$ red edges, we have that at least one red edge joins $v$ to a vertex that is earlier in $\ll$.
If $v$ is not matched in $M$, then we select one such edge to be violet.
Obviously, we get that every vertex that is incident to $h$ red edges is incident to at least one violet edge.
Further, every vertex is incident to at most one violet edge that joins the vertex to a vertex that is earlier in $\ll$.
As the violet edges satisfy the two conditions mentioned above, we get that the resulting partition of edges is a $(d+1,h-1)$\nobreakdash-decomposition of $G$.
\end{proof}

Observe that in the proof of \cref{lem:super} the matching $M$ can be found in time $\Oh{m^{1+\oh{1}}}$~\cite{ChenKLPGS22,ChenKLPGS23}, and the second step can easily be implemented to run in linear time. Thus the procedure of \cref{lem:super} runs in time $\Oh{m^{1+\oh{1}}}$.

Equipped with \cref{lem:super}, we are now ready to recall and prove \cref{thm:super}.

\thmsuper*
\begin{proof}
	The inequality $d \ge k$ is necessary by \cref{exm:d}, and the inequality $d+h \ge 2k-1$  is necessary by \cref{exm:dh}. For sufficiency, observe first that every $(d,h-1)$\nobreakdash-decomposition and every $(d-1,h)$\nobreakdash-decomposition is also a $(d,h)$\nobreakdash-decomposition. Thus, it is enough to prove the theorem for $d+h=2k-1$.
	Let $G$ be a $k$\nobreakdash-arboric graph.
	As $d \ge k$ and $d+h = 2k-1$, we have that $h \le k-1$.
	By \cref{thm:main}, we have a $(k,k-1)$\nobreakdash-decomposition of $G$.
	By $k-1-h$ applications of \cref{lem:super}, we can transform the $(k,k-1)$\nobreakdash-decomposition first into a $(k+1,k-2)$\nobreakdash-decomposition, then into a $(k+2,k-3)$\nobreakdash-decomposition, and so on.
	As $k+(k-1-h)=2k-1-h=d$ and $k-1-(k-1-h)=h$, after $k-1-h$ applications we obtain a $(d,h)$\nobreakdash-decomposition of~$G$.
\end{proof}

\section{Algorithm}\label{sec:algorithm}
The main result of this section is \cref{lem:algo_decomposition} which shows that \cref{lem:main} can be implemented to run in polynomial time.
Quite unexpectedly, the algorithm is simple, as it does not require any insight into the proof of \cref{lem:friendly_high_rank}.

Fix an integer $k \ge 1$ and an object $(G,\redc)$ with $G=(V,E)$, $n=\norm{V}$, and $m=\norm{E}$.
Recall that the object~$(G,\redc)$ is positive when the potential of every non-empty subset of vertices is at least $2$. This can be rewritten as
\begin{align}
    \forall_{\emptyset \neq U \subseteq V}\,2\cdot\norm{E[U]} \le 2k\cdot\norm{U} - 2\redc(U)- 2\text{.} \label{eq:positive}
\end{align}
The condition must hold for every non-empty subset $U$.
Following the approach of Picard and Queyranne~\cite{PicardQ1982}, we give an equivalent formulation branching on the choice of a vertex $u\in U$:
\begin{align}
    \forall_{u \in V}\,\forall_{U \subseteq V, u \in U}\,2\cdot\norm{E[U]} \le 2k\cdot\norm{U} - 2\redc(U)- 2\text{.} \label{eq:positive2}
\end{align}

Now, let us define a network flow problem that follows the flow formulation for bipartite matchings and will be used as a subproblem in the algorithm.
Based on $k$, $G$, and $\redc$, we construct a four layer network $\mathcal{N}_{k,G,\redc}$.
In the first layer, we have a single source node $s$.
In the second layer, we have a single node $e$ for each edge $e \in E$.
We connect the source node $s$ with every node $e$ in the second layer using a directed link of capacity $2$.
In the third layer, we have a single node $v$ for each vertex $v \in V$.
We connect every node $e=\set{v,w}$ in the second layer with each of the nodes $v$ and $w$ using a directed link of infinite capacity.
Finally, in the fourth layer, there is a single sink node~$t$.
We connect each node $v$ in the third layer with the sink node using a directed link of capacity $2k-2\redc(v)$. Notice that this capacity (by the definition of an object) is positive. 
It is an easy observation that $\mathcal{N}_{k,G,\redc}$ admits a flow of size $2\cdot \norm{E}$ if and only if the following condition holds:
\[
    \forall_{U \subseteq V}\,2\cdot \norm{E[U]} \le 2k\cdot\norm{U} - 2\cdot\redc(U)\text{.}
\]
This condition is very similar to the condition for $(G,\redc)$ to be a positive object.
We can exploit this similarity to decide if a given object is positive.


For the first application of the network flow approach, let $\mathcal{N}'_{k,G,\redc}$ be the network obtained from $\mathcal{N}_{k,G,\redc}$ by increasing the capacity of every link connecting the source node with a node in the second layer from $2$ to $2+\frac{1}{m+1}$.
Then, $\mathcal{N}'_{k,G,\redc}$ admits a flow of size $\brac{2+\frac{1}{m+1}}\cdot \norm{E}$ if and only if the following holds:
\[
\forall_{U \subseteq V}\,\brac{2+\frac{1}{m+1}}\cdot \norm{E[U]} \le 2k\cdot\norm{U} - 2\cdot\redc(U)\text{.}
\]
As $2\cdot\norm{E[U]}$, $2k\cdot\norm{U}$, and $2\cdot\redc(U)$ are even integers, and 
\[
2\cdot\norm{E[U]} < \brac{2+\frac{1}{m+1}}\cdot \norm{E[U]} < 2\cdot\norm{E[U]} + 1
\] 
whenever $E[U]\neq\emptyset$, we get that $\mathcal{N}'_{k,G,\redc}$ admits a flow of size $\brac{2+\frac{1}{m+1}}\cdot \norm{E}$ if and only if
\[
\forall_{\emptyset \neq U \subseteq V}\,2\cdot \norm{E[U]}+2 \le 2k\cdot\norm{U} - 2\cdot\redc(U)\text{.}
\]
This is exactly condition~(\ref{eq:positive}).
Using the recently developed algorithms for maximum flow~\cite{ChenKLPGS22,ChenKLPGS23}, we can check if $\mathcal{N}'_{k,G,\redc}$ admits a flow of size $\brac{2+\frac{1}{m+1}}\cdot \norm{E}$ in time $\Oh{m^{1+\oh{1}}}$ and we get the following.
\defdecproblem
{\DPpositive}
{Integer $k \ge 1$ and an object~$(G,\redc)$}
{Is $(G,\redc)$ a positive object?}
\begin{observation}\label{obs:algo_positive}
    There is an algorithm that, for an integer $k \ge 1$ and an object $(G,\redc)$,
    decides $\DPpositive \langle k,G,\redc \rangle$ in time $\Oh{m^{1+\oh{1}}}$.
\end{observation}

For the second application of the network flow approach, we give a different formulation based on condition~(\ref{eq:positive2}) that will be more useful for our algorithm.
Let $u \in V$ be any vertex, and $\mathcal{N}^{[u]}_{k,G,\redc}$ be the network obtained from $\mathcal{N}_{k,G,\redc}$ by decreasing the capacity of the link connecting the node~$u$ to the sink node $t$ by $2$.
As $\redc(u) \le k-1$, the capacity of that link in the resulting network is non-negative.
Then, $\mathcal{N}^{[u]}_{k,G,\redc}$ admits a flow of size $2\cdot \norm{E}$ if and only if the following two conditions hold:
\[
\begin{array}{l}
    \forall_{\emptyset \neq U \subseteq V}\,2\cdot \norm{E[U]} \le 2k\cdot\norm{U} - 2\cdot\redc(U)\text{,}\\
    \forall_{\emptyset \neq U \subseteq V, u \in U}\,2\cdot \norm{E[U]} \le 2k\cdot\norm{U} - 2\cdot\redc(U) - 2\text{.}
\end{array}
\]
This equivalence together with condition~(\ref{eq:positive2}) gives the following observation.
\begin{observation}\label{obs:algo_positive2}
    For an integer $k \ge 1$, an object $(G,\redc)$ with $G=(V,E)$ is positive if and only if for every vertex $v \in V$, the network $\mathcal{N}^{[v]}_{k,G,\redc}$ admits a flow of size $2\cdot \norm{E}$.
\end{observation}

The goal of this section is to give an efficient algorithm that solves the following problem.
\defoptproblem
{\OPdecomposition}
{Integer $k \ge 1$ and a positive unobstructed object $(G,\redc)$}
{A $(k,k-1)$\nobreakdash-decomposition of $(G,\redc)$}
\begin{lemma}\label{lem:algo_decomposition}
    There is an algorithm that, for an integer $k \ge 1$ and a positive unobstructed object $(G,\redc)$ with $m$ edges,
    constructs $\OPdecomposition \langle k,G,\redc \rangle$ in time $\Oh{m^3}$.
\end{lemma}
\begin{proof}
Our proof of \cref{lem:main} gives the following algorithm to construct a $(k,k-1)$\nobreakdash-decomposition.
For any positive unobstructed object~$(G,\redc)$ on input, \cref{lem:friendly_high_rank} guarantees that there is at least one reduction available in $(G,\redc)$.
Algorithm finds and applies any reduction of Type~\ref{red:easy} if it exists.
Otherwise, Algorithm finds and applies any reduction of Type~\ref{red:hard} if it exists.
Otherwise, Algorithm finds and applies any reduction of Type~\ref{red:tree} which is guaranteed to exist.
After applying the reduction, we get a smaller positive unobstructed object, and we can repeat the process until we get an empty object.
Let $(H_1,\redc_1), (H_2,\redc_2),\ldots,(H_s,\redc_s)$ be a sequence of consecutive objects where $(H_1,\redc_1)=(G,\redc)$, and each $(H_{i+1},\redc_{i+1})$ is constructed from $(H_i,\redc_i)$ by the application of a reduction selected by the algorithm.
Let $m_i$ denote the number of edges in $H_i$ for every $i=1,2,\ldots,s$.
The algorithm terminates when it reaches the empty object $(H_s,\redc_s)$.

Finding a reduction of Type~\ref{red:easy} is easy as it requires only a simple check for the degree of each vertex.
Finding a reduction of Type~\ref{red:hard} is more involved, as it requires deciding if a given neighbor is a safe neighbor or a dangerous neighbor.
Finding a reduction of Type~\ref{red:tree} is again easy, as we already know that there is no reduction of Type~\ref{red:hard} available, i.e., that every neighbor of a vertex with positive score is a dangerous neighbor.
It is easy to see that $w$ is a safe neighbor of $v$ if and only if the object obtained from $(G,\redc)$ by removing the edge $\set{v,w}$ and increasing the red-count of~$w$ by one is positive.
Thus, by \cref{obs:algo_positive}, we can check if $w$ is a safe neighbor of $v$ in time $\Oh{m^{1+\oh{1}}}$.
This approach gives an algorithm that runs in time $\Oh{m^{3+\oh{1}}}$.

In what follows, we exploit the decremental setting of our queries that allows us to perform every single check for a safe neighbor in time $\Oh{m}$.
To allow for efficient checks for a safe neighbor, we maintain, at every step~$i$, an integral maximum flow in the network $\mathcal{N}_{k,H_i,\redc_i}$; such a flow exists as all finite capacities in the network are integers.
Before the first step, we compute an integral maximum flow in $\mathcal{N}_{k,G,\redc}$ using the algorithm of~\cite{ChenKLPGS22,ChenKLPGS23} in time $\Oh{m^{1+\oh{1}}}$.
When we apply reductions, we have to perform two basic modifications of an object: edge removal and red-count increase.
Now, we discuss how these basic modifications influence the flow network.
When an edge $e$ is removed from the graph, we have to remove the node $e$ from the flow network and the flow decreases by $2$.
This is easily done in constant time.
When we increase the red-count of a vertex $w$, we have to decrease by $2$ the capacity of the link connecting the node $w$ with the sink node $t$, and the maintained flow may become infeasible.
In this case, we first restore feasibility by cancelling at most $2$ units of flow along flow paths containing this link.
As the red-count of a vertex is increased only in a reduction of Type~\ref{red:hard}, \cref{lem:friendly_reductions} gives that the resulting network still admits a flow saturating all links leaving the source, and hence the maximum flow is higher than the value of the maintained flow by at most $2$ units.
Thus, we can find a new maximum flow by performing at most two flow augmentations in time $\Oh{m}$.
As we perform at most $m$ edge removals and at most $m$ red-count increases, the total time spent on maintaining the flow is $\Oh{m^2}$.

It remains to show how to use the maintained flow to perform a single check for a safe neighbor in time $\Oh{m}$.
Let $(H_i,\redc_i)$ be the current object with $m_i$ edges.
We maintain a flow of size $2\cdot m_i$ in $\mathcal{N}_{k,H_i,\redc_i}$.
Let $v$ be a vertex in $H_i$ and $w$ be a neighbor of $v$ in $H_i$.
Let $(H',\redc')$ be the object obtained from $(H_i,\redc_i)$ by removing edge $\set{v,w}$ and increasing the red-count of $w$ by one.
We already know that $w$ is a safe neighbor of $v$ if and only if $(H',\redc')$ is a positive object with $m' = m_i - 1$ edges.
Thus, by \cref{obs:algo_positive2}, it is enough to check if $\mathcal{N}^{[x]}_{k,H',\redc'}$ admits a flow of size $2m'$ for every vertex~$x$ in $H'$.
As $(H_i,\redc_i)$ is a positive object, we know that if some subset $U$ of vertices has potential lower than $2$ in $(H',\redc')$, then $\pot_{H',\redc'}(U) = 0$ and $w \in U$.
Thus, we only need to check if $\mathcal{N}^{[w]}_{k,H',\redc'}$ admits a flow of size $2m'$.
Notice that the flow network $\mathcal{N}^{[w]}_{k,H',\redc'}$ together with a maximum flow can be easily constructed in time $\Oh{m}$ by making a copy of the flow network $\mathcal{N}_{k,H_i,\redc_i}$, removing a single node, decreasing the capacity of a single link by $4$ and performing a constant number of flow augmentations.
As the algorithm performs $\Oh{m^2}$ checks, the total running time is $\Oh{m^3}$.
\end{proof}

Equipped with \cref{lem:algo_decomposition}, we are now ready to recall and prove \cref{lem:algorithm}.

\lemalgorithm*
\begin{proof}
	Let $G$ be a $k$\nobreakdash-arboric graph with $m$ edges. 
	Recall that every $(d,h-1)$\nobreakdash-decomposition and every $(d-1,h)$\nobreakdash-decomposition is also a $(d,h)$\nobreakdash-decomposition, and therefore, it is enough to prove the assertion for $d+h=2k-1$. 
	By \cref{lem:algo_decomposition}, the algorithm constructs a $(k,k-1)$\nobreakdash-decomposition of $G$ in time $\Oh{m^3}$. Next, as in the proof of \cref{thm:super}, the algorithm uses $k-1-h$ applications of \cref{lem:super} to convert it into a $(d,h)$\nobreakdash-decomposition. Since the procedure of \cref{lem:super} runs in time $\Oh{m^{1+\oh{1}}}$, the total running time is $\Oh{m^3}$.
\end{proof}

\section{Hardness}\label{sec:hardness}

Lima, Rautenbach, Souza and Szwarcfiter~\cite{LimaRSS2017} show that \DPdec{1}{1} is \NP-complete.
They use a very robust reduction from \DPham.
In particular, the completeness holds even when the input graph is 
$2$-connected bipartite cubic planar simple graph.
Using similar ideas, for all positive integers $d$ and $h$, we show that \DPdec{d}{h} is \NP-complete.

\thmhard*
\begin{proof}
    Notice that checking if a given graph is $d$\nobreakdash-degenerate, or $h$\nobreakdash-bounded-degree can be done in linear time.
    Thus, \DPdec{d}{h} is in \NP.
    The hardness proof goes by induction on $d$.
    For the base case, we present, for any positive integer $h$, a reduction from the following variant of \DPham to \DPdec{1}{h}.
\defdecproblem
{\DPfixham{r}}
{An $r$-regular graph $G=(V,E)$ and an edge $e$ in $E$}
{Is there a Hamiltonian Cycle in $G$ that includes $e$?}
    It follows from the original works~\cite{GareyJT1976,GareyJS1976} on the complexity of \DPham that \DPfixham{r} is \NP-complete for every $r \ge 3$.
    Now, we fix $r=h+2$, and show a reduction to \DPdec{1}{h}.
    Assume we are given instance $\langle G, e \rangle$ and let $u$ and $v$ be the two vertices in $e$.
    Let $G'$ be a graph obtained from $G$ by removing edge $e$.
    In graph $G'$, vertices $u$ and $v$ are of degree $r-1=h+1$ each.
    Every other vertex in $G'$ is of degree $r=h+2$.
    We claim that $\langle G, e \rangle$ is a \yes-instance if and only if $G'$ admits a $(1,h)$\nobreakdash-decomposition.
    See \cref{fig:hardness_base}, that illustrates the reduction for $h=1$.

    \begin{figure}[h]
        \centering
        \begin{tikzpicture}[scale=0.8]
            \node[draw, circle] (u1) at (0:4) {$u$};
            \node[draw, circle] (u2) at (30:4) {$v$};
            \node[draw, circle] (u3) at (60:4) {\phantom{$u$}};
            \node[draw, circle] (u4) at (90:4) {\phantom{$u$}};
            \node[draw, circle] (u5) at (120:4) {\phantom{$u$}};
            \node[draw, circle] (u6) at (150:4) {\phantom{$u$}};
            \node[draw, circle] (v6) at (180:4) {\phantom{$u$}};
            \node[draw, circle] (v5) at (210:4) {\phantom{$u$}};
            \node[draw, circle] (v4) at (240:4) {\phantom{$u$}};
            \node[draw, circle] (v3) at (270:4) {\phantom{$u$}};
            \node[draw, circle] (v2) at (300:4) {\phantom{$u$}};
            \node[draw, circle] (v1) at (330:4) {\phantom{$u$}};
            \draw[dotted, thick] (u2) -- (u1);
            \draw[blue, thick] (u2) -- (u3) -- (u4) -- (u5) -- (u6) -- (v6) -- (v5) -- (v4) -- (v3) -- (v2) -- (v1) -- (u1);
            \draw[red, thick] (u1) -- (u6);
            \draw[red, thick] (v1) -- (v6);
            \draw[red, thick] (u2) -- (v2);
            \draw[red, thick] (u4) -- (v3);
            \draw[red, thick] (u3) -- (v4);
            \draw[red, thick] (u5) -- (v5);
        \end{tikzpicture}
        \caption{The reduction from \DPham to \DPdec{1}{1}. A cubic graph $G$ has a Hamiltonian cycle that includes edge $e=\set{u,v}$ if and only if the graph $G'$ obtained by removing $e$ admits a $(1,1)$\nobreakdash-decomposition, i.e.\ decomposition into a forest and a matching.
        Blue edges span a forest, red edges span a matching, and the dotted edge is the removed edge $e$.
        }

        \label{fig:hardness_base}
    \end{figure}
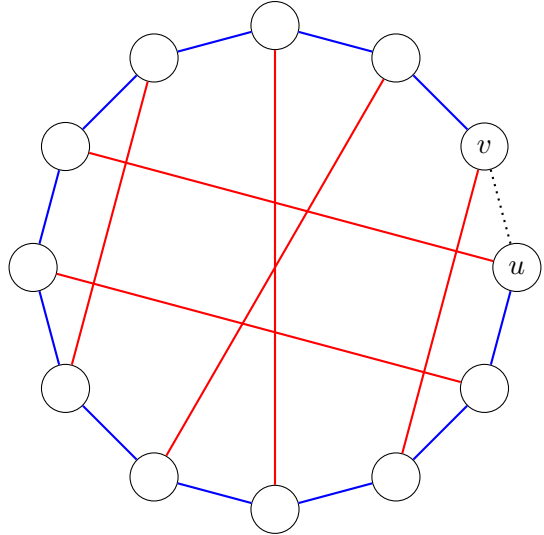

    For the completeness of the reduction, see that if there is a Hamiltonian cycle in $G$ that includes~$e$, then the other edges of the cycle span a Hamiltonian $u$-$v$ path which is $1$\nobreakdash-degenerate, and the edges not in the cycle span an $h$\nobreakdash-regular graph.

    For the soundness of the reduction, let $E_1$, $E_2$ be a $(1,h)$\nobreakdash-decomposition of $G'$.
    We color the edges in $E_1$ blue and the edges in $E_2$ red.
    Since every vertex other than $u$ and $v$ has degree $h+2$ in $G'$ and has at most $h$ incident red edges, it has blue degree at least $2$.
    Similarly, $u$ and $v$ have blue degree at least $1$.
    Thus, there are no isolated vertices in the blue forest, and there are only two leaves in the blue forest.
    Thus, the blue forest spanned by $E_1$ is a Hamiltonian $u$-$v$ path.
    Adding edge $e$ to $E_1$ gives a Hamiltonian cycle in $G$.

    For every positive integer $h$, we have that \DPdec{1}{h} is \NP-complete.
    Now, for the inductive step, let us consider any positive integers $d$ and $h$ and show a reduction from \DPdec{d}{h} to \DPdec{d+1}{h}.
    For an instance $\langle G \rangle$ of \DPdec{d}{h}, we construct a graph $G'$ by adding an apex vertex $a$ to $d+h+2$ disjoint copies $G_1, G_2, \ldots, G_{d+h+2}$ of the graph $G$, see \cref{fig:hardness_step}.
    We claim that $\langle G \rangle$ is a \yes-instance
    if and only if 
    $G'$ admits a $(d+1,h)$-decomposition.

    \begin{figure}[h]
        \centering
        \begin{tikzpicture}[scale=0.8]
            \node[draw, circle] (a) at (-2,0) {$a$};
            \node[draw, rectangle] (G1) at (0,0) {$G_1$};
            \node[draw, rectangle] (G2) at (2,0) {$G_2$};
            \node[rectangle] (G3) at (4,0) {$\ldots$};
            \node[draw, rectangle] (G4) at (6,0) {$G_{d+h+1}$};
            \node[draw, rectangle] (G5) at (9,0) {$G_{d+h+2}$};
            \draw (a.north east) edge[blue, bend left=45] (G1.north west);
            \draw (a.north east) edge[blue, bend left=45] (G1.north);
            \draw (a.north east) edge[blue, bend left=45] (G1.north east);
            \draw (a.north east) edge[blue, bend left=45] (G2.north west);
            \draw (a.north east) edge[blue, bend left=45] (G2.north);
            \draw (a.north east) edge[blue, bend left=45] (G2.north east);
            \draw (a.north east) edge[blue, bend left=45] (G4.north west);
            \draw (a.north east) edge[blue, bend left=45] (G4.north);
            \draw (a.north east) edge[blue, bend left=45] (G4.north east);
            \draw (a.north east) edge[blue, bend left=45] (G5.north west);
            \draw (a.north east) edge[blue, bend left=45] (G5.north);
            \draw (a.north east) edge[blue, bend left=45] (G5.north east);
        \end{tikzpicture}
        \caption{The reduction from \DPdec{d}{h} to \DPdec{d+1}{h}. 
        Putting the apex vertex $a$ first in the order and coloring all edges incident to $a$ blue, we obtain a $(d+1,h)$\nobreakdash-decomposition from combined $(d,h)$\nobreakdash-decompositions of the copies of $G$.
        \label{fig:hardness_step}
        }
    \end{figure}
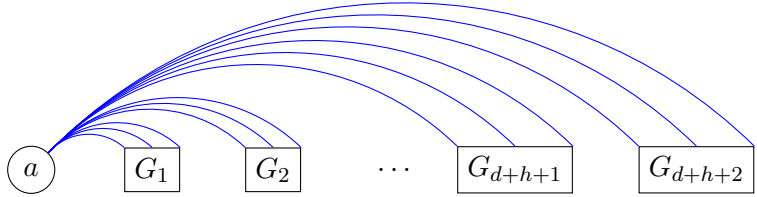

    For the completeness of the reduction, see that if there is a $(d,h)$-decomposition of $G$, then one can construct a $(d+1,h)$\nobreakdash-decomposition of $G'$ by joining decompositions of the copies of $G$, and adding all the edges incident to the apex vertex $a$ to the $d$\nobreakdash-degenerate graph.
    The resulting graph is $(d+1)$-degenerate, as the vertex $a$ can be easily added on the first position in a $(d+1)$\nobreakdash-degenerate order.
    
    For the soundness of the reduction, let $E_1$, $E_2$ be a $(d+1,h)$-decomposition of $G'$.
    We color the edges in $E_1$ blue and the edges in $E_2$ red.
    Let $\ll$ be a $(d+1)$-degenerate order of the blue graph.
    As the apex vertex $a$ has at most $h$ incident red edges, at most $h$ copies of $G$ contain an endpoint of a red edge incident with the apex vertex $a$.
    Similarly, there are at most $d+1$ blue edges incident to $a$ that join $a$ to vertices that are earlier in $\ll$. 
    As there are $d+h+2$ disjoint copies of $G$ in $G'$, at least one copy of $G$ has all vertices after $a$ in $\ll$ and no red edge joining this copy with $a$.
    Let $G_i$ be such a copy of $G$ for some $i$ in $\{1,\ldots,d+h+2\}$.
    Trivially, each vertex in $G_i$ has at most $h$ incident red edges, as it has at most $h$ incident red edges in $G'$.
    Further, each vertex $v$ in $G_i$ has at most $d+1$ incident blue edges to vertices that are earlier in $\ll$.
    As $a$ is earlier than $v$ in $\ll$, and the edge $\set{a,v}$ is blue, we get that $\ll$ restricted to vertices of $G_i$ is a $d$\nobreakdash-degenerate order for blue edges.
    Therefore, the restriction of the decomposition to $G_i$ is a $(d,h)$\nobreakdash-decomposition of $G$.
\end{proof}

As shown in~\cite{LimaRSS2017}, \DPdec{1}{1} is \NP-complete even for bipartite planar graphs.
In particular, it is \NP-complete for $2$-arboric graphs.
This fact combined with the reduction used in the proof of \cref{thm:hard} guarantees that \DPdec{2}{1} is \NP-complete even for $3$-arboric graphs.
On the other hand, \cref{thm:super} guarantees that every $2$-arboric graph $G$ admits a $(2,1)$-decomposition.
We conjecture that there are sharp complexity thresholds for \DPdec{k}{k-1} -- the problems are in \cP{} for $k$-arboric graphs, and \NP-complete for $(k+1)$-arboric graphs or even for smaller classes of graphs that properly contain $k$-arboric graphs.

\section*{Use of AI} 
The mathematical content of this article was developed entirely by the authors, without the use of artificial intelligence tools.

\section*{Acknowledgements}
	Bart{\l}omiej Bosek, Grzegorz Gutowski and Jakub Przyby{\l}o have discussed and attempted to solve the problem for $k=2$ with many researchers including Marcin Anholcer, Jarosław Grytczuk, Rafał Pyzik, Oriol Serra, Lluis Vena Cros, and Mariusz Zając.
	\newline Micha{\l} Laso\'n thanks Green Caffè Nero and A4 Krak\'{o}w-Wroc{\l}aw for the inspiring environment.

\bibliography{paper}
\end{document}